\documentclass[
10pt
]{article}

\usepackage[%
left=1in,
right=1in,
bottom=1in,
top=1in,
footskip=0.3in,
]{geometry}
\usepackage{color}
\definecolor{LinkColor}{rgb}{0,0,1}
\definecolor{LinkColor2}{rgb}{1,0,0}
\definecolor{lbcolor}{rgb}{0.85,0.85,0.85}
\definecolor{FrameColor}{rgb}{0.85,0.85,0.85}

\usepackage{lineno}

\newcommand*\patchAmsMathEnvironmentForLineno[1]{%
	\expandafter\let\csname old#1\expandafter\endcsname\csname #1\endcsname
	\expandafter\let\csname oldend#1\expandafter\endcsname\csname end#1\endcsname
	\renewenvironment{#1}%
	{\linenomath\csname old#1\endcsname}%
	{\csname oldend#1\endcsname\endlinenomath}}%
\newcommand*\patchBothAmsMathEnvironmentsForLineno[1]{%
	\patchAmsMathEnvironmentForLineno{#1}%
	\patchAmsMathEnvironmentForLineno{#1*}}%
\AtBeginDocument{%
	\patchBothAmsMathEnvironmentsForLineno{equation}%
	\patchBothAmsMathEnvironmentsForLineno{align}%
	\patchBothAmsMathEnvironmentsForLineno{flalign}%
	\patchBothAmsMathEnvironmentsForLineno{alignat}%
	\patchBothAmsMathEnvironmentsForLineno{gather}%
	\patchBothAmsMathEnvironmentsForLineno{multline}%
}

\usepackage[ngerman, english]{babel}
\usepackage[utf8]{inputenc}
\usepackage{enumerate}
\usepackage[normalem]{ulem}
\usepackage{setspace}

\usepackage{multicol}
\usepackage[babel,french=guillemets,german=swiss]{csquotes}
\usepackage[center]{caption}
\usepackage{calligra}

\usepackage{array}
\usepackage{booktabs}
\usepackage{framed}
\usepackage{rotating}
\usepackage{longtable}
\usepackage{multirow}
\usepackage{tabularx}
\newcolumntype{L}[1]{>{\raggedright\arraybackslash}p{#1}} 
\newcolumntype{C}[1]{>{\centering\arraybackslash}p{#1}} 
\newcolumntype{R}[1]{>{\raggedleft\arraybackslash}p{#1}} 

\usepackage{graphicx}
\usepackage{placeins}

\usepackage{amsmath}
\usepackage{amssymb}
\usepackage{dsfont}
\usepackage{nicefrac}
\usepackage{empheq}
\allowdisplaybreaks
\usepackage{amsthm}
\usepackage{upgreek}

\usepackage[%
pdftitle={Titel},%
pdfauthor={Autor},%
pdfcreator={LaTeX, LaTeX with hyperref and KOMA-Script},
pdfsubject={Betreff}, 
pdfkeywords={Keywords}
]{hyperref} 

\hypersetup{%
	colorlinks	=true,
	linkcolor	=LinkColor,%
	anchorcolor	=LinkColor,%
	citecolor	=LinkColor2,%
	filecolor	=LinkColor,%
	menucolor	=LinkColor,%
	pagecolor	=LinkColor,%
	urlcolor	=LinkColor,%
}

\newtheorem{theorem}{Theorem}
\newtheorem{lemma}[theorem]{Lemma}
\newtheorem{proposition}[theorem]{Proposition}
\newtheorem{corollary}[theorem]{Corollary}
\newtheorem{definition}[theorem]{Definition}
\newtheorem{remark}[theorem]{Remark}
\newtheorem{example}[theorem]{Example}

\makeatletter
\renewenvironment{proof}[1][\proofname]{%
	\par\pushQED{\qed}\normalfont%
	\topsep6\p@\@plus6\p@\relax
	\trivlist\item[\hskip\labelsep\bfseries#1\@addpunct{.}]%
	\ignorespaces
}{%
	\popQED\endtrivlist\@endpefalse
}
\makeatother

\makeatletter
\renewcommand\paragraph{\@startsection{paragraph}{4}{\z@}%
	{1ex \@plus1ex \@minus.2ex}%
	{-1em}%
	{\normalfont\normalsize\itshape}}
\renewcommand\subparagraph{\@startsection{paragraph}{4}{\z@}%
	{1ex \@plus1ex \@minus.2ex}%
	{-1em}%
	{\normalfont\normalsize\itshape}}
\makeatother

\newcommand{\norm}[1]{\ensuremath\lVert #1 \rVert}
\newcommand{\norml}[2]{\ensuremath\lVert #2 \rVert_{L^{#1}}}
\newcommand{\normL}[2]{\ensuremath\lVert #2 \rVert_{\L^{#1}}}
\newcommand{\normw}[3]{\ensuremath\lVert #3 \rVert_{W^{#1,#2}}}

\newcommand{\normh}[2]{\ensuremath\lVert #2 \rVert_{H^{#1}}}
\newcommand{\normH}[2]{\ensuremath\lVert #2 \rVert_{\H^{#1}}}

\def\A{\mathcal A}
\def\B{\mathcal B}
\def\S{\mathcal S}
\def\R{\mathbb R}
\def\CR{\mathcal R}
\def\C{\mathcal C}
\def\N{\mathbb N}
\def\n{\mathbf n}

\def\U{\mathbb U}

\def\UR{\mathbb{U}_R}

\def\u{{\bar u}}
\def\L{\mathbf L}
\def\H{\mathbf H}
\def\h{\mathds{h}}
\def\v{\mathbf{v}}
\def\w{\mathbf{w}}
\def\T{\mathbf{T}}
\def\I{\mathbf{I}}
\def\D{{\mathbf{D}}}

\def\V{{\mathcal V}}

\def\F{{\mathbf{F}}}
\def\G{{\mathbf{G}}}

\def\tu{{\tilde u}}
\def\tv{{\tilde\v}}
\def\tvarphi{{\tilde\varphi}}
\def\tmu{{\tilde\mu}}
\def\tsigma{{\tilde\sigma}}
\def\tp{{\tilde p}}

\def\tw{{\tilde\w}}
\def\tphi{{\tilde\phi}}
\def\ttau{{\tilde\tau}}
\def\trho{{\tilde\rho}}

\def\OmegaT{{\Omega_T}}
\def\GammaT{{\Gamma_T}}

\def\intO{\int_\Omega}
\def\intOT{\int_\OmegaT}

\def\dx{\;\mathrm dx}

\def\dt{\;\mathrm dt}
\def\ds{\;\mathrm ds}
\def\dxt{\;\mathrm d(x,t)}

\def\dS{\;\mathrm dS}

\def\del{\partial}

\def\delt{\partial_{t}}

\def\deln{\partial_\n}

\def\d{{\mathrm{d}}}

\def\grad{\nabla}
\def\laplace{\Delta}
\def\divergence{\textnormal{div}}

\def\sminus{\hspace{-2pt}-\hspace{-2pt}}
\def\scdot{\hspace{-2pt}\cdot\hspace{-2pt}}
\def\scolon{\hspace{-2pt}:\hspace{-2pt}}

\def\tand{\quad\text{and}\quad}
\def\twith{\quad\text{with}\quad}

\def\scdot{\hspace{-2pt}\cdot\hspace{-2pt}}

\def\wto{\rightharpoonup}

\def\itema{\item[\textnormal{(a)}]}
\def\itemb{\item[\textnormal{(b)}]}
\def\itemc{\item[\textnormal{(c)}]}
\def\itemd{\item[\textnormal{(d)}]}

\def\itemi{\item[\textnormal{(i)}]}
\def\itemii{\item[\textnormal{(ii)}]}

\def\blk{\color{black}}

\begin{document}
	
%
%
	
\begin{center}	
	\LARGE{Optimal control theory and advanced optimality conditions\\ for a diffuse interface model of tumor growth}
\end{center}

\bigskip

\begin{center}	
	\normalsize{Matthias Ebenbeck and Patrik Knopf}\\[1mm]
	\textit{Department of Mathematics, University of Regensburg, 93053 Regensburg, Germany}\\[1mm]
	\texttt{Matthias.Ebenbeck@ur.de, Patrik.Knopf@ur.de}\\[5mm]
	\textit{This is a preprint version of the paper. Please cite as:}\\
	M. Ebenbeck and P. Knopf,
	ESAIM Control Optim.~Calc.~Var. 26(71):38, 2020.\\
	\url{https://doi.org/10.1051/cocv/2019059}
\end{center}

\vfill

\begin{abstract}
	We investigate a distributed optimal control problem for a diffuse interface model for tumor growth. The model consists of a Cahn--Hilliard type equation for the phase field variable, a reaction diffusion equation for the nutrient concentration and a Brinkman type equation for the velocity field. These PDEs are endowed with homogeneous Neumann boundary conditions for the phase field variable, the chemical potential and the nutrient as well as a \grqq no-friction" boundary condition for the velocity. The control represents a medication by cytotoxic drugs and enters the phase field equation. The aim is to minimize a cost functional of standard tracking type that is designed to track the phase field variable during the time evolution and at some fixed final time. We show that our model satisfies the basics for calculus of variations and we present first-order and second-order conditions for local optimality. Moreover, we present a globality condition for critical controls and we show that the optimal control is unique on small time intervals. \\[1ex]

	\noindent\textit{Keywords:} Optimal control with PDEs, calculus of variations, tumor growth, Cahn--Hilliard equation, Brinkman equation, first-order necessary optimality conditions, second-order sufficient optimality conditions, uniqueness of globally optimal solutions.\\[1ex]
	
	\noindent\textit{MSC Classification:} 35K61, 76D07, 49J20, 92C50.
\end{abstract}

\vfill

%
%

\section{Introduction}
The evolution of cancer cells is influenced by a variety of biological mechanisms like, e.g., cell-cell adhesion or mechanical stresses, see \cite{BearerEtAl}. Although cancer is one of the most common causes of death, the knowledge about underlying processes is still at an unsatisfying level. Due to the complexity of the evolution, experimental and medical research may not be sufficient to predict growth and to establish general treatment strategies. Thus, it is of high importance to develop biologically realistic and predictive mathematical models to identify the influence of different growth factors and mechanisms and to set up individual treatment. \\[1ex]
In the recent past, diffuse interface models gained much interest (e.g., \cite{FrigeriGrasselliRocca,GarckeLamSitkaStyles,HawkinsZeeKristofferOdenTinsley,HilhorstKampmannNguyenZee}) and some of them seem to compare well with clinical data, see \cite{AgostiEtAl,BearerEtAl,FrieboesEtAl}. Typically, these models are derived from balance equations for mass and momentum and they incorporate exchange of mass and momentum between the phases. Mechanisms like chemotaxis, necrosis, angiogenesis and apoptosis can be included or effects due to viscoelasticity and stress, see e.g. \cite{ChristiniLiLowengrubWise,GarckeLamNuernbergSitka,OdenTinsleyHawkins}.\\[1ex]

\pagebreak[4]
\noindent In general, living biological tissues behave like viscoelastic fluids, see \cite{Chaplain}. Since relaxation times of elastic materials are rather short (see \cite{Fung}), it is reasonable to consider Stokes flow as an approximation of those tissues, see e.g., \cite{FranksKing}. Indeed, many authors used Stokes flow to describe the tumor as a viscous fluid, see \cite{FranksKing2,Friedman}. In classical tumor growth models, velocities are modeled with the help of Darcy's law. In these models the velocity is assumed to be proportional to the pressure gradient caused by the birth of new cells and by the deformation of the tissue, see \cite{ByrneChaplain2,Grennspan}. Brinkman's law is an interpolation between the viscous fluid and the Darcy-type models, see e.g., \cite{PerthameVauchelet,ZhengWiseCristini}.\\[1ex]

\noindent\textbf{Introduction of the model.}\newline
In this paper, we consider the following model: Let $\Omega\subset \R^d$ with $d=2,3,$ be a bounded domain. For a fixed final time $T>0$, we write $\OmegaT:=\Omega\times(0,T)$. By $\n $ we denote the outer unit normal on $\del\Omega$ and $\deln g \coloneqq \grad g\cdot \n $ denotes the outward normal derivative of the function $g$ on $\Gamma$. Our state system is given by
\begin{subequations}
	\label{EQ:CHB}
	\begin{empheq}[left=\text{(CHB)}\empheqlbrace]{align}
		\label{EQ:CHB1}
		\divergence(\v) &= (P\sigma-A)\h(\varphi) &&\text{in}\; \OmegaT,\\
		\label{EQ:CHB2}
		-\divergence(\T(\v,p)) + \nu\v &= (\mu+\chi\sigma)\grad\varphi &&\text{in}\; \OmegaT,\\
		\label{EQ:CHB3}
		\delt\varphi + \divergence(\varphi\v) &= m\laplace\mu + (P\sigma-A-u)\h(\varphi) &&\text{in}\; \OmegaT,\\
		\label{EQ:CHB4}
		\mu &= -\epsilon\laplace\varphi + \epsilon^{-1}\psi'(\varphi) -\chi\sigma &&\text{in} \; \OmegaT,\\
		\label{EQ:CHB5}
		-\laplace\sigma + \h(\varphi)\sigma &= b(\sigma_B-\sigma) &&\text{in}\; \OmegaT,\\[2mm]
		\label{EQ:CHB6}
		\deln\mu=\deln\varphi=\deln\sigma &= 0 &&\text{in}\; \GammaT,\\
		\label{EQ:CHB7}
		\T(\v,p)\n &= 0 &&\text{in}\; \GammaT,\\[2mm]
		\label{EQ:CHB8}
		\varphi(0) &= \varphi_0	&&\text{in}\; \Omega,
	\end{empheq}
\end{subequations}
where the viscous stress tensor is defined by
\begin{equation}
\label{definition_stress_tensor}\T(\v,p) \coloneqq 2\eta \D\v+\lambda\divergence(\v)\I  - p\I ,
\end{equation}
and the symmetric velocity gradient is given by
\begin{equation*}
\D\v\coloneqq \tfrac{1}{2}(\grad\v+\grad\v^T).
\end{equation*}
In \eqref{EQ:CHB}, we denote by $\varphi$ the difference of volume fractions of tumor tissue and healthy tissue where $\{x\in\Omega:\varphi(x) = 1\}$ represents the region of unmixed tumor tissue, $\{x\in\Omega:\varphi(x) = -1\}$ stands for the surrounding healthy tissue. The volume-averaged velocity of the mixture and the pressure are denoted by $\v$ and $p$, respectively. The tumor consumes an unknown species $\sigma\in [0,1]$ acting as a nutrient like e.g. oxygen or glucose. Furthermore, $\mu$ denotes the chemical potential associated with $\varphi$. The mobility is represented by a positive constant $m$, the diffuse interface thickness is proportional to a small parameter $\epsilon>0$ and the constant $\nu>0$ is related to the fluid porosity. The shear and bulk viscosities are given by a positive constant $\eta$ and a non-negative constant $\lambda$, respectively. The non-negative constants $P$, $A$ and $\chi$ represent the proliferation rate, the apoptosis rate and the chemotaxis parameter, respectively. By $\sigma_B$, we denote the nutrient concentration in a preexisting vasculature and $b$ is a positive constant. Hence, the term $b(\sigma_B-\sigma)$ models the nutrient supply from the blood vessels if $\sigma_B>\sigma$ and the nutrient transport away from the domain for $\sigma_B<\sigma$. The term $-u\h(\varphi)$ in \eqref{EQ:CHB3} models the elimination of tumor cells by cytotoxic drugs and the function $u$ will act as our control. 
Since it does not play any role in the analysis, we set $\epsilon = 1$.
\\[1ex]
We investigate the following distributed optimal control problem:
\begin{align}
\notag\text{Minimize}\quad I(\varphi,\mu,\sigma,\v,p,u)&:= 
\frac{\upalpha_0} 2 \norml{2}{\varphi(T)-\varphi_f}^2
+ \frac{\upalpha_1} 2 \norm{\varphi-\varphi_d}_{L^2(\Omega_T)}^2
+ \frac{\kappa} 2\norm{u}_{L^2(\Omega_T)}^2
\end{align}
subject to the control constraint
\begin{align}
u\in \U:=\big\{ u\in L^2(L^2) \;\big\vert\; \bar{a}(x,t) \le u(x,t) \le \bar{b}(x,t) \;\;\text{for almost every}\; (x,t)\in\OmegaT \big\}
\end{align}
for box-restrictions $\bar{a},\bar{b}\in L^2(L^2)$ and the state system \eqref{EQ:CHB}. Here, $\upalpha_0,\upalpha_1$ and $\kappa$ are nonnegative constants. \\[1ex] 
The optimal control problem can be interpreted as the search for a strategy how to supply a medication such that a desired evolution $\varphi_d$ and a therapeutic target $\varphi_f$ are achieved in the best possible way without causing harm to the patient (expressed by both the control constraint and the last term in the cost functional).
The ratio between the parameters $\upalpha_0$, $\upalpha_1$ and $\kappa$ can be adjusted according to the importance of the individual therapeutic targets.\\[1ex] 
In the case when $\h(-1)=0$, the term $-u\h(\varphi)$ models the elimination of tumor cells by a supply of cytotoxic drugs represented by the control $u$. This specific control term has been investigated in \cite{EbenbeckKnopf} and also in \cite{GarckeLamRocca} where a simpler model was studied in which the influence of the velocity $\v$ is neglected. \blk However, in some situations it may be more reasonable to control, for instance, the evolution at the interface and one has to use a different form for $\h(\cdot)$, see Remark \ref{REM:PSI_H} below. Therefore, we allow $\h(\cdot)$ to be rather general. \\[1ex]
%
%
\noindent\textbf{Summary of our main results.}\newline
In Section 3, we prove the existence of a control-to-state operator that maps any admissible control $u\in\U$ onto a corresponding unique strong solution of the state equation \eqref{EQ:CHB}. Furthermore, we show that this control-to-state operator is Lipschitz-continuous, Fréchet differentiable and satisfies a weak compactness property. In particular, we establish the fundamental requirements for calculus of variations. \\[1ex] 
In Section 4, we investigate the adjoint system. Its solution, that is called the adjoint state or the costate, is an important tool in optimal control theory as it provides a better description of optimality conditions. We prove the existence of a control-to-costate operator which maps any admissible control onto its corresponding adjoint state. Then, we show that this control-to-costate operator is Lipschitz continuous and Fréchet differentiable. \\[1ex]
Eventually, in Section 5, we investigate the above optimal control problem. First, we show that there exists at least one globally optimal solution. After that, we establish first-order necessary conditions for local optimality. These conditions are of great importance for possible numerical implementations as they provide the foundation for many computational optimization methods. We also present a second-order sufficient condition for strict local optimality, a globality criterion for critical controls and a uniqueness result for the optimal control on small time intervals. \\[1ex]
\noindent\textbf{Comparison with the results established in \cite{EbenbeckKnopf}.}\newline
We want to outline the main features and novelties of our work, especially compared to the results in \cite{EbenbeckKnopf}. The main difference lies in the equation for the nutrient concentration. In \cite{EbenbeckKnopf} the equation
\begin{equation}
\label{EQ:EBKN}
-\laplace\sigma + \h(\varphi)\sigma = 0\quad\text{in }\Omega_T,\qquad \deln\sigma = K(1-\sigma)\quad\text{on }\GammaT
\end{equation}
was used where $K$ denotes some positive boundary permeability constant. However, in this paper, we use equation \eqref{EQ:CHB5} endowed with a homogeneous Neumann boundary condition (as studied in \cite{GarckeLamRocca}) to describe the nutrient distribution. \\[1ex]
Apart from the fact that \eqref{EQ:CHB5} might be more reasonable in some situations from a modeling point of view, it provides some advantages with regard to analysis and optimal control theory. In particular, Lipschitz continuity and Fréchet differentiability of both the control-to-state operator and the control-to-costate operator can be established in much better function spaces. As a consequence, we can prove that the cost functional $J$ is twice continuously differentiable which was not possible in \cite{EbenbeckKnopf}. This enables us to establish second-order conditions for (strict) local optimality.\\[1ex]
Finally, we want to comment on the fact that active transport is neglected in \eqref{EQ:EBKN}. In principle, it is possible to decouple the effects of chemotaxis and active transport, see, e.g., \cite{GarckeLamSitkaStyles}. If the ratio between nutrient diffusion and active transport timescale is quite small, a non-dimensionalization argument together with the decoupling justifies the fact that active transport may be neglected, cf. \cite{GarckeLam3}. \\[1ex]
Moreover, as mentioned above, we present a globality criterion for critical controls as well as a uniqueness result for the optimal control on small time intervals. These new conditions have neither been established in \cite{EbenbeckKnopf} nor for related models in the literature. However, we are convinced that analogous conditions could also be proved for the optimal control problem in \cite{EbenbeckKnopf}.\\[1ex]
\noindent\textbf{Related results in the literature.}\newline
Finally, we want to mention further works where optimal control problems for tumor models are studied.
Results on optimal control problems for tumor models based on ODEs are investigated in \cite{SchaettlerLedzewicz2,Oke,SchaettlerLedzewicz,Swan}. In the context of PDE-based control problems we refer to \cite{Benosman} where a tumor growth model of advection-reaction-diffusion type is considered.
There are various papers analyzing optimal control problems for Cahn--Hilliard equations (e.g., \cite{ColliFarshbaf,HintermuellerWegner,Signori,Signori2,Signori3}). Furthermore, control problems for the convective Cahn--Hilliard equation where the control acts as a velocity were investigated in \cite{ColliGilardiSprekels,GilardiSprekels,ZhaoLiu1,ZhaoLiu2} whereas in \cite{Biswas,FrigeriRoccaSprekels} the control enters in the momentum equation of a Cahn--Hilliard--Navier--Stokes system. As far as control problems for Cahn--Hilliard-based models for tumor growth are considered, there are only a few contributions where an equation for the nutrient is included in the system. In \cite{ColliGilardiRoccaSprekels}, the authors investigated an optimal control problem consisting of a Cahn--Hilliard-type equation coupled to a time-dependent reaction-diffusion equation for the nutrient, where the control acts as a right-hand side in this nutrient equation. The model they considered was firstly proposed in \cite{HawkinsZeeKristofferOdenTinsley} and later well-posedness and existence of strong solutions were established in \cite{FrigeriGrasselliRocca}. However, effects due to velocity are not included in their model and mass conservation holds for the sum of tumor and nutrient concentrations. A similar model has been analyzed in \cite{Cavaterra} both from the viewpoint of optimal long-term and finite-time treatment of medication. Furthermore, we want to cite the paper \cite{GarckeLamRocca} about an optimal control problem of treatment time where the control represents a medication of cytotoxic drugs and enters the phase field equation in the same way as ours. Although their nutrient equation is non-stationary, some of the major difficulties do not occur since the velocity is assumed to be negligible ($\v=\mathbf 0$). Lastly, we want to mention the work \cite{Sprekels-Wu}, where the authors tackled an optimal control problem for a Cahn--Hilliard--Darcy model describing tumor growth in two space dimensions. However, this model neglects the influence of nutrients and poses a compatibility condition for the source term in the divergence equation resulting from a different boundary condition for the velocity. In particular, the source term does not depend on additional variables such as $\varphi$. \\[1ex]

\section{Preliminaries}
At first, we want to fix some notation: For any (real) Banach space $X$, its corresponding norm is denoted by $\norm{\,\cdot\,}_X$. We write $X^\ast$ to denote the dual space of $X$ and $\langle \cdot{,}\cdot \rangle_X$ to denote duality pairing between $X^\ast$ and $X$. If $X$ is endowed with an inner product, this product is denoted by $(\cdot{,}\cdot)_X$. The scalar product of two matrices is defined by
\begin{equation*}
\mathbf{A}:\mathbf{B}\coloneqq \sum_{j,k=1}^{d}a_{jk}b_{jk}\quad\text{for } \mathbf{A}=(a_{ij})_{1\le i,j\le d},\; \mathbf{B}=(b_{ij})_{1\le i,j\le d} \in\R^{d\times d}.
\end{equation*}
For the standard Lebesgue and Sobolev spaces with $1\leq p\leq \infty$ and $k>0$, we use the notation $L^p\coloneqq L^p(\Omega)$ and $W^{k,p}\coloneqq W^{k,p}(\Omega)$. Their corresponding norms are denoted by $\norml{p}{\,\cdot\,}$ and $\normw{k}{p}{\,\cdot\,}$. If $p=2$ we write $H^k\coloneqq W^{k,2}$ and $\normh{k}{\,\cdot\,} \coloneqq \normw{k}{2}{\,\cdot\,}$. Sometimes, we will also write $\norm{\,\cdot\,}_p$ instead of $\norm{\,\cdot\,}_{L^p(L^p)}$. Moreover, we write $\mathbf{L}^p$, $\mathbf{W}^{k,p}$ and $\mathbf{H}^k$ to denote the corresponding spaces of vector or matrix valued functions. For Bochner spaces, we use the notation $L^p(X)\coloneqq L^p(0,T;X)$ for any Banach space $X$ and $p\in [1,\infty]$.  For the dual space $X^\ast$ of a Banach space $X$, we define the (generalized) mean value by 
\begin{equation*}
v_{\Omega}\coloneqq \frac{1}{|\Omega|}\intO v\dx\quad\text{for } v\in L^1, \quad v_{\Omega}^\ast\coloneqq \frac{1}{|\Omega|}\langle v{,}1\rangle_X\quad\text{for } v\in X^\ast.
\end{equation*}
Furthermore, we define the following function spaces:
\begin{equation*}
L_0^2\coloneqq \{w\in L^2\colon w_{\Omega}=0\},\quad (H^1)_0^\ast\coloneqq \{f\in (H^1)^\ast\colon f_{\Omega}^\ast =0\},\quad H_\n^2\coloneqq \{w\in H^2\colon \deln w = 0 \text{ on } \del\Omega\}.
\end{equation*}
We remark that the Neumann-Laplace operator $-\laplace_N\colon H^1\cap L_0^2\to (H^1)_0^\ast$ is positive definite and self-adjoint. In particular, due to the Lax-Milgram theorem and the Poincaré inequality,
its inverse operator $(-\laplace_N)^{-1}\colon (H^1)_0^\ast\to H^1\cap L_0^2$ is well-defined and we write $v\coloneqq (-\laplace_N)^{-1}f$ for ${f\in (H^1)_0^\ast}$ if $v_{\Omega}=0$ and 
\begin{equation*}
-\laplace v = f\text{ in }\Omega,\quad\deln v =0\text{ on }\del\Omega.
\end{equation*}
The embeddings $H_\n^2\hookrightarrow H^1\hookrightarrow L^2\simeq (L^2)^\ast\hookrightarrow (H^1)^\ast\hookrightarrow (H_\n^2)^\ast$ are continuous. Therefore we can identify $\langle u{,}v\rangle_{H^1}=(u{,}v)_{L^2}$, $\langle u{,}w\rangle _{H^2} = (u{,}w)_{L^2}$ for all $u\in L^2,~ v\in H^1$ and $w\in H_\n^2$.\\[1ex]
Eventually, we introduce the function spaces
\begin{align*}
\V_1 &\coloneqq \big(H^1(L^2)\cap L^{\infty}(H^2)\cap L^2(H^4)\big) \times \big(L^{\infty}(L^2)\cap L^2(H^2)\big)\times  L^{\infty}(H^2)\times  L^8(\H^2)\times L^8(H^1),\\
\V_2 &\coloneqq L^8(\L^2)\times L^8(H^1)\times L^2(L^2)\times \big( L^{\infty}(L^2)\cap L^2(H^2)\big)\times L^{\infty}(L^2),\\
V_3&\coloneqq \big(H^1((H^1)^*)\cap L^{\infty}(H^1)\cap L^2(H^3)\big) \times L^2(H^1)\times  L^{2}(H^2)\times  L^2(\H^2)\times L^2(H^1),\\
\V_4&\coloneqq H^1\times L^2(L^{6/5})\times  L^2(H^1)\times L^2(L^2)\times L^{2}(\H^1)\times L^2(\L^2),
\end{align*}
endowed with their standard norms.\\[1ex]


\noindent Moreover, we make the following assumptions which we will use in the rest of this paper:\\[1ex]
\noindent\textbf{Assumptions.}
\begin{enumerate}
	\item[(A1)] The domain $\Omega\subset\R^d,~d=2,3,$ is bounded with $C^4-$boundary $\Gamma:=\del\Omega$. Moreover the initial datum $\varphi_0\in H^2_\n$ and $\sigma_B\in C([0,T];L^2)$ are given functions. 
	\item[(A2)] The constants $T$, $\eta$, $\nu$, $m$, $b$ are positive and the constants $P$, $A$, $\lambda$, $\chi$ are nonnegative.
	\item[(A3)] The non-negative function $\h$ belongs to $C^3_b(\R)$, i.e., $\h$ is bounded, three times continuously differentiable and its first, second and third-order derivatives are bounded. Without loss of generality, we assume that $|\h|\le 1$.
	\item[(A4)] $\psi$ is the smooth double-well potential, i.e., $\psi(s):=\frac 1 4 (s^2-1)^2$ for all $s\in\R$.
\end{enumerate}


\pagebreak[2]

\begin{remark}\label{REM:PSI_H} $\;$
\begin{enumerate}
\itema In principal, it would be possible to consider more general potentials $\psi(\cdot)$. However, since the double-well potential is the classical choice for Cahn--Hilliard-type equations (apart from singular potentials like the logarithmic or double-obstacle potential) and to avoid being too technical, we focus on the above choice for $\psi$ in this work. 
\itemb For the function $\h(\cdot)$, there are two choices which are quite popular in the literature. In, e.g., \cite{GarckeLam3, GarckeLamSitkaStyles}, the choice for $\h$ is given by
\begin{equation*}
\h(\varphi) = \max\Big\{0,\min\left\{1,\tfrac{1}{2}(1+\varphi)\right\}\Big\}\quad\forall \varphi\in\R,
\end{equation*}
satisfying $\h(-1)= 0$, $\h(1)=1$. Other authors preferred to assume that $\h$ is only active on the interface, i.e., for values of $\varphi$ between $-1$ and $1$, which motivates functions of the form
\begin{equation*}
\h(\varphi) = \max\Big\{0,\tfrac{1}{2}(1-\varphi^2)\Big\}\quad\text{or}\quad \h(\varphi) = \tfrac{1}{2}\Big(\cos\big(\pi \min \big\{1,\max\{\varphi,-1\} \big\}\big) + 1\Big),
\end{equation*}
see, e.g., \cite{HilhorstKampmannNguyenZee,KahleLam}. Surely, we would have to use regularized versions of these choices to fulfill \textnormal{(A3)}.
\itemc We want to point out that the analysis for singular potentials is quite delicate, even for the uncontrolled version of \eqref{EQ:CHB}, since source terms are present. In contrast to, for example, \cite{Signori2} where the estimates for $\mu$ are obtained from the energy due to an additional regularization term $\alpha\delt\mu$ (with $\alpha>0$) in \eqref{EQ:CHB3}, our analysis is based on an estimate for the mean value of $\mu$. Clearly, if $\psi(\cdot)$ is singular, such an estimate is quite hard to obtain. Following the arguments, for instance, in \cite{BloweyElliot1}, it is crucial to prove that $\frac{1}{|\Omega|}\intO \varphi(t)\dx\in (-1,1)$ for almost every $t\in (0,T)$. \\[1ex]
Since our model does not have the property of mass conservation for the phase field $\varphi$, we would require additional assumptions on the source terms and an argument such as used in \cite{GarckeLamStyles}. However, this argument does not work when using source terms, for example, as in \cite{HilhorstKampmannNguyenZee,KahleLam} (see (b)) in combination with the double obstacle potential, which can be seen as follows: If there was a solution to \eqref{EQ:CHB} that fulfills $(\varphi(t_*))_{\Omega}=1$ for some $t_*\in (0,T)$, then $(\varphi(t))_{\Omega}=1$ for all $t\in [t_*,T]$.

\end{enumerate}
\end{remark}

\section{The control-to-state operator and its properties}

We consider the system \eqref{EQ:CHB} as presented in the introduction. The first step is to define a set of controls that are admissible for our problem. Then we show that each of these admissible controls induces a unique strong solution (the so-called \textbf{state}) of the system \eqref{EQ:CHB}. Thus, we can define a control-to-state-operator which maps any admissible control onto its corresponding state. We show that this operator has several important properties that are essential for calculus of variations: It is Lipschitz-continuous, Fréchet-differentiable and weakly compactness in some suitable sense.

\subsection{The set of admissible controls}
The set of admissible controls is defined as follows:
\begin{definition}
	Let $\bar{a},\bar{b}\in L^2(L^2)$ be arbitrary fixed functions with $\bar{a}\le \bar{b}$ almost everywhere in $\OmegaT$. Then the set
	\begin{align}
		\U:=\big\{ u\in L^2(L^2) \;\big\vert\; \bar{a}(x,t) \le u(x,t) \le \bar{b}(x,t) \;\;\text{for almost every}\; (x,t)\in\OmegaT \big\}
	\end{align}
	is referred to as the \textbf{set of admissible controls}. Its elements are called \textbf{admissible controls}.
\end{definition}
Note that this box-restricted set of admissible controls $\U$ is a non-empty, bounded subset of the Hilbert space $L^2(L^2)$ since for all $u\in\U$,
\begin{align}
	\norm{u}_{L^2(L^2)} < \norm{\bar{a}}_{L^2(L^2)} + \norm{\bar{b}}_{L^2(L^2)} + 1 =: R.
\end{align}
This means that 
\begin{align}
	\U \subsetneq \UR \twith \UR := \big\{ u\in L^2(L^2) \;\big\vert\; \norm{u}_{L^2(L^2)} <R \big\}.
\end{align}
Obviously, the set $\U$ is also convex and closed in $L^2(L^2)$. Therefore, it is weakly sequentially compact (see \cite[Thm.\,2.11]{troeltzsch}). 

\subsection{Strong solutions and uniform bounds}

We can show that the system \eqref{EQ:CHB} has a unique strong solution for every control $u\in\UR$:

\begin{proposition}
	\label{THM:EXS}
	Let $ u\in\UR$ be arbitrary. Then, there exists a strong solution quintuplet $(\varphi_u,\mu_u,\sigma_u,\v_u,p_u)\in \V_1$ of \eqref{EQ:CHB}.
	 Moreover, every strong solution $(\varphi_u,\mu_u,\sigma_u,\v_u,p_u)$ satisfies the following bounds that are uniform in $u$: 
	\begin{equation}
		\label{THM:EXS:EST}\norm{(\varphi_u,\mu_u,\sigma_u,\v_u,p_u)}_{\V_1} \le C_1,	
	\end{equation}
	where $C_1>0$ is a constant that depends only on on $R$, $\Omega$ and $T$ and the system parameters.
\end{proposition}

\pagebreak[3]

\begin{proof}
	The assertion follows with slight modifications in the proof of \cite[Thm.\,4]{EbenbeckKnopf}. We will sketch the main differences in the following:
	\paragraph{Step 1:} Testing \eqref{EQ:CHB5} with $\sigma$, using \eqref{EQ:CHB6}, the non-negativity of $\h(\cdot)$ and Hölder's and Young's inequalities, we obtain
	\begin{equation*}
	\intO |\grad\sigma|^2\dx + b\intO |\sigma|^2\d x \leq \frac{b}{2}\intO |\sigma|^2 + |\sigma_B|^2\dx,
	\end{equation*}
	meaning
	\begin{equation}
	\label{proof_ex_1} \normh{1}{\sigma}\leq C\norml{2}{\sigma_B}.
	\end{equation}
	Then, with exactly the same arguments as in \cite[Proof of Thm.\,4]{EbenbeckKnopf} it follows that
	\begin{align}
	\nonumber&\norm{\varphi}_{H^1((H^1)^*)\cap L^{\infty}(H^1)\cap L^4(H^2)\cap L^2(H^3) } + \norm{\sigma}_{L^4(H^1)} + \norm{\mu}_{L^2(H^1)\cap L^4(L^2)}\\
	\label{proof_ex_5}&\quad + \norm{\divergence(\varphi\v)}_{L^2(L^2)}  + \norm{\v}_{L^{\frac{8}{3}}(\H^1)}+ \norm{p}_{L^2(L^2)}\leq C,
	\end{align}
	with a constant $C$ depending only on the system parameters and on $\Omega$, $T$ and $R$.
	\paragraph{Step 2:} Now, we want to establish higher order estimates.
	Using elliptic regularity theory, the assumptions on $\h(\cdot)$ and $\sigma_B$, \eqref{EQ:CHB5}-\eqref{EQ:CHB7}, \eqref{proof_ex_1} and \eqref{proof_ex_5}, it is easy to check that
	\begin{equation}
	\label{proof_ex_6} \norm{\sigma}_{L^{\infty}(H^2)}\leq C.
	\end{equation}
	Together with the boundedness of $\h(\cdot)$ and the Sobolev embedding $H^2\subset L^{\infty}$, this implies
	\begin{equation}
	\label{proof_ex_7}\norm{\sigma}_{L^{\infty}(L^{\infty})} + \norm{\divergence(\v)}_{L^{\infty}(L^{\infty})} + \norm{(P\sigma-A)\h(\varphi)}_{L^{\infty}(L^{\infty})}\leq C.
	\end{equation}
	Testing \eqref{EQ:CHB3} with $\laplace^2\varphi$,  \eqref{EQ:CHB3} with $m\laplace^3\varphi$, integrating by parts and summing the resulting identities, we obtain
	\begin{align}
	\nonumber \frac{\d}{\d t}\frac{1}{2}\intO |\laplace\varphi|^2\dx + m\intO |\laplace^2\varphi|^2\dx &= \intO \big( (P\sigma-A-u)\h(\varphi)-\divergence(\varphi\v)\big) \laplace^2\varphi\d x\\
	\label{proof_ex_8}&\quad + \intO m\laplace(\psi'(\varphi) - \chi\sigma)\laplace^2\varphi\dx.
	\end{align}
	Due to Hölder's and Young's inequalities and \eqref{proof_ex_5}-\eqref{proof_ex_7}, the Sobolev embedding $\H^1\subset \L^6$ and elliptic estimates, it follows that
	\begin{align}
	\label{proof_ex_9}
	\left|\intO \big( (P\sigma-A-u)\h(\varphi)-\divergence(\varphi\v)-m\chi \laplace \sigma\big)\laplace^2\varphi\dx \right|&\leq C\norml{2}{u}^2 + C\normH{1}{v}^2\left(1+\norml{2}{\laplace\varphi}^2\right)+ \frac{m}{4}\norml{2}{\laplace^2\varphi}^2.
	\end{align}
	Now, we observe that
	\begin{equation*}
	\laplace(\psi'(\varphi)) = \psi'''(\varphi)|\grad\varphi|^2 + \psi''(\varphi)\laplace\varphi.
	\end{equation*}
	Using Hölder's, Young's and Gagliardo-Nirenberg's inequalities, the assumptions on $\psi(\cdot)$, elliptic regularity theory and  \eqref{proof_ex_5}, this implies
	\begin{align}
	\nonumber \left|\intO m\laplace(\psi'(\varphi))\laplace^2\varphi\d x\right| &\leq C\left(1+\norml{\infty}{\varphi}^4 + \norml{\infty}{\varphi}^2\normL{3}{\grad\varphi}^2\right)(1 + \norml{2}{\laplace\varphi}^2) + \frac{m}{2}\norml{2}{\laplace^2\varphi}^2\\
	\label{proof_ex_10}&\leq C\left(1+\normh{3}{\varphi}^2\right)\left(1 + \norml{2}{\laplace\varphi}^2\right) + \frac{m}{2}\norml{2}{\laplace^2\varphi}^2.
	\end{align}
	Using \eqref{proof_ex_9}-\eqref{proof_ex_10} in \eqref{proof_ex_8}, recalling \eqref{proof_ex_5} and using elliptic regularity theory, a Gronwall argument yields
	\begin{equation}
	\label{proof_ex_11}\norm{\varphi}_{L^{\infty}(H^2)} + \norm{\varphi}_{L^2(H^4)}\leq C. 
	\end{equation}
	Then, a comparison argument in \eqref{EQ:CHB4} yields
	\begin{equation}
	\label{proof_ex_12}\norm{\mu}_{L^{\infty}(L^2)\cap L^2(H^2)}\leq C.
	\end{equation}
	A further comparison argument in \eqref{EQ:CHB3} yields
	\begin{equation}
	\label{proof_ex_13}\norm{\delt\varphi}_{L^2(L^2)}\leq C.
	\end{equation}
	Using \eqref{proof_ex_5}-\eqref{proof_ex_7}, \eqref{proof_ex_11}-\eqref{proof_ex_12}, the assumptions on $\h(\cdot)$ and Gagliardo-Nirenberg's inequality, it is easy to check that
	\begin{equation*}
	\norm{(P\sigma-A)\h(\varphi)}_{L^8(H^1)}\leq C,\qquad \norm{(\mu+\chi\sigma)\grad\varphi}_{L^8(\L^2)}\leq C.
	\end{equation*}
	Due to \cite[Lemma 1.5]{EbenbeckGarcke}, this implies
	\begin{equation}
	\norm{\v}_{L^8(\H^2)} + \norm{p}_{L^8(H^1)}\leq C,
	\end{equation}
	which completes the proof.
\end{proof}	

By means of interpolation theory, we can conclude that the $\varphi$-component of a strong solution quintuplet has a representative that is continuous on $\OmegaT$.

\begin{corollary}
	Let $u\in\UR$ and $\varphi_0\in H^2_\n(\Omega)$ be arbitrary and let $(\varphi_u,\mu_u,\sigma_u,\v_u,p_u)$ denote the strong solution of the system \eqref{EQ:CHB}. Then, $\varphi_u$ has the following additional properties:
	\begin{gather*}
		 \varphi_u\in C([0,T];H^2), \quad \varphi_u\in C(\overline{\OmegaT}) \twith \norm{\varphi_u}_{C([0,T];H^2)\cap C(\overline{\OmegaT})} \le C_2 
	\end{gather*}
	for some constant $C_2>0$ that depends only on the system parameters and on $R$, $\Omega$, $\Gamma$ and $T$.
\end{corollary}

\begin{proof}
	The assertion can be established by an interpolation argument. The proof proceeds completely analogously to the proof of \cite[Cor. 5]{EbenbeckKnopf}.
\end{proof}

\bigskip

Furthermore, we can show that any control $u\in\UR$ induces a unique strong solution of the system \eqref{EQ:CHB}:

\begin{theorem}
\label{THM:UNS} Let $u\in\UR$ and $\varphi_0\in H^2_\n(\Omega)$ be arbitrary and let $(\varphi_{u},\mu_{u},\v_{u},\sigma_{u},p_{u})$ denote the corresponding strong solution as given by Proposition \ref{THM:EXS}. Then, this strong solution is unique.
\end{theorem}

\begin{proof}
	Let $u,\tu\in\UR$ be arbitrary and let $C$ denote a generic nonnegative constant that depends only on $R$, $\Omega$, $\Gamma$ and $T$ and may change its value from line to line. For brevity, we set
	\begin{align*}
	(\varphi,\mu,\v,\sigma, p):=(\varphi_{u},\mu_{u},\v_{u},\sigma_{u},p_{u})-(\varphi_\tu,\mu_\tu,\v_\tu,\sigma_\tu,p_\tu).
	\end{align*}
	where $(\varphi_{u},\mu_{u},\v_{u},\sigma_{u},p_{u})$ and $(\varphi_\tu,\mu_\tu,\v_\tu,\sigma_\tu,p_\tu)$ are strong solutions of \eqref{EQ:CHB} to the controls $u$ and $\tu$. In particular, this means that both strong solutions satisfy the initial condition \eqref{EQ:CHB8}, i.e., $\varphi_u(\cdot,0)=\varphi_\tu(\cdot,0)=\varphi_0$ holds almost everywhere in $\Omega$.\\[1ex]
	
	\pagebreak[2]
	\noindent Then, the following equations are satisfied:
	\begin{subequations}
		\label{EQ:DIFF}
		\begin{align}
		\label{EQ:DIFF1}
		\divergence(\v) &= P\sigma\h(\varphi_u) + (P\sigma_\tu-A)(\h(\varphi_u)-\h(\varphi_\tu))&&\text{in}\; \OmegaT,\\
		\label{EQ:DIFF2}
		-\divergence(\T(\v,p)) + \nu\v &= (\mu+\chi\sigma)\grad\varphi_u + (\mu_\tu+\chi\sigma_\tu)\grad\varphi &&\text{in}\; \OmegaT,\\
		\nonumber
		\delt\varphi + \divergence(\varphi_u\v) + \divergence(\varphi\v_\tu) &= m\laplace\mu + P\sigma\h(\varphi_u) + (P\sigma_\tu-A)(\h(\varphi_u)-\h(\varphi_\tu))\\
		\label{EQ:DIFF3}&\quad -(u\h(\varphi_u)-\tu\h(\varphi_\tu))&&\text{in}\; \OmegaT,\\
		\label{EQ:DIFF4}
		\mu &= -\laplace\varphi + (\psi'(\varphi_u)-\psi'(\varphi_\tu)) -\chi\sigma &&\text{in} \; \OmegaT,\\
		\label{EQ:DIFF5}
		-\laplace\sigma + b\sigma + \h(\varphi_u)\sigma &= -\sigma_\tu(\h(\varphi_u)-\h(\varphi_\tu)) &&\text{in}\; \OmegaT,\\[2mm]
		\label{EQ:DIFF6}
		\deln\mu=\deln\varphi=\deln\sigma &= 0 &&\text{in}\; \GammaT,\\
		\label{EQ:DIFF7}
		\T(\v,p)\n &= 0 &&\text{in}\; \GammaT,\\[2mm]
		\label{EQ:DIFF8}
		\varphi(0) &= 0	&&\text{in}\; \Omega.
		\end{align}
	\end{subequations}
	Now, we will show that $\norm{(\varphi,\mu,\v,\sigma, p)}_{\V_1} = 0$ if $u=\tu$. The argumentation is split into two steps:
	\paragraph{Step 1:}
	In \cite{EbenbeckKnopf}, it has been shown that the following inequalities hold: For any $\delta>0$ and all $u,\tu\in\U$,
	\begin{align}
		\label{EQ:LIP1}
		\big|\big( u\h(\varphi_u) - \tu\h(\varphi_\tu){,}\varphi \big)_{L^2}\big| 
			&\le C\,\norml{2}{u-\tu}^2 + C\, \norml{2}{ \varphi}^2 + C\,\norml{2}{\tu}\normh{1}{\varphi}^2,\\
		\label{EQ:LIP2}
		\big|\big( u\h(\varphi_u) - \tu\h(\varphi_\tu){,}\laplace\varphi \big)_{L^2}\big| 
			&\le C\delta^{-1}\, \norml{2}{u-\tu}^2 + C\delta^{-1}\, \norml{2}{\tu}^2 \normh{1}{\varphi}^2
				+ 2\delta \norml{2}{\laplace\varphi}^2 + \delta \normL{2}{\grad\laplace\varphi}^2.
	\end{align}
	Now, multiplying \eqref{EQ:DIFF5} with $\sigma$, integrating by parts and using \eqref{EQ:DIFF6}, it follows that
	\begin{equation*}
	\intO |\grad\sigma|^2 + b|\sigma|^2 + \h(\varphi_u)|\sigma|^2\dx = -\intO \sigma_\tu(h(\varphi_u) - \h(\varphi_\tu))\sigma\dx.
	\end{equation*}
	Using the assumptions on $\h(\cdot)$, Proposition \ref{THM:EXS} and Hölder's and Young's inequalities, it is therefore easy to check that
	\begin{equation}
	\label{EQ:LIP3}\normh{1}{\sigma}\leq C\norml{2}{\varphi}.
	\end{equation}
	Then, we can follow the arguments in \cite[Sec. 5]{EbenbeckGarcke2} to deduce that
	\begin{equation}
	\label{EQ:LIP4}\norm{\varphi}_{H^1((H^1)^*)\cap L^{\infty}(H^1)\cap L^2(H^3)} + \norm{\mu}_{L^2(H^1)} + \norm{\sigma}_{L^2(H^1)} + \norm{\v}_{L^2(\H^1)} + \norm{p}_{L^2(L^2)}\leq C\norm{u-\tu}_{L^2(L^2)}. 
	\end{equation}

	\paragraph{Step 2:} We now prove higher order estimates.
	Using elliptic regularity theory, Proposition \ref{THM:EXS}, \eqref{EQ:LIP3}-\eqref{EQ:LIP4} and the assumptions on $\h(\cdot)$, it is easy to check that
	\begin{equation}
	\label{EQ:LIP5}\norm{\sigma}_{L^{\infty}(H^2)}\leq C\norm{u-\tu}_{L^2(L^2)}.
	\end{equation} 
	Multiplying \eqref{EQ:DIFF3} with $\laplace^2\varphi$ and inserting the expression for $\mu$ given by \eqref{EQ:DIFF4}, we obtain
	\begin{align}
	\nonumber \frac{\d}{\dt}\frac{1}{2}\intO |\laplace\varphi|^2\dx + m\intO |\laplace^2\varphi|^2\dx &= \intO \big( P\sigma\h(\varphi_u) + (P\sigma_\tu-A)(\h(\varphi_u)-\h(\varphi_\tu))\big)\laplace^2\varphi\dx\\
	\nonumber &\quad -\intO \big(\divergence(\varphi_u\v) + \divergence(\varphi\v_\tu)\big)\laplace^2\varphi\dx \\
	\nonumber &\quad + m\intO \laplace\big((\psi'(\varphi_u)-\psi'(\varphi_\tu)) -\chi\sigma\big)\laplace^2\varphi\dx \\
	\label{EQ:LIP6}&\quad -\intO (u\h(\varphi_u)-\tu\h(\varphi_\tu))\laplace^2\varphi\dx,
	\end{align}
	where we used that
	\begin{equation*}
	\intO \delt\varphi\,\laplace^2\varphi\dx = \frac{\d}{\dt}\frac{1}{2}\intO |\laplace\varphi|^2\dx\quad \forall \varphi\in H^2(L^2)\cap L^2(H^4),\quad\deln\varphi = \deln\laplace\varphi = 0.
	\end{equation*}
	
	\pagebreak[2]
	
	\noindent Using Proposition \ref{THM:EXS}, \eqref{EQ:LIP4}-\eqref{EQ:LIP5} together with Hölder's and Young's inequalities, it follows that
	\begin{align}
	\nonumber &\left|\intO \big( P\sigma\h(\varphi_u) + (P\sigma_\tu-A)(\h(\varphi_u)-\h(\varphi_\tu))- \divergence(\varphi_u\v) - \divergence(\varphi\v_\tu)-m\chi\laplace\sigma\big)\laplace^2\varphi\dx \right|\\
	\label{EQ:LIP7}&\quad \leq C\left(\normH{1}{\v}^2 + \normh{1}{\varphi}^2\normH{2}{\v_\tu}^2+ \normh{2}{\sigma}^2 + \norml{2}{\varphi}^2\right) + \frac{m}{8}\norml{2}{\laplace^2\varphi}^2.
	\end{align}
	Using the continuous embedding $H^2\subset L^{\infty}$, Proposition \ref{THM:EXS}, \eqref{EQ:LIP4}, the assumptions on $\h(\cdot)$ and the elliptic estimate
	\begin{equation*}
	\normh{2}{\varphi}\leq C\left(\norml{2}{\varphi} + \norml{2}{\laplace\varphi}\right)\quad\forall \varphi\in H_{\n}^2,
	\end{equation*}
	we obtain 
	\begin{align}
	\nonumber\left|\intO (u\h(\varphi_u)-\tu\h(\varphi_\tu))\laplace^2\varphi\dx\right| & = \left|\intO \big((u-\tu)\h(\varphi_u) +\tu(\h(\varphi_u)-h(\varphi_\tu))\big)\laplace^2\varphi\dx\right|\\
	\nonumber &\leq C\left(\norml{2}{u-\tu} + \norml{2}{\tu}(\norml{2}{\varphi} + \norml{2}{\laplace\varphi})\right)\norml{2}{\laplace^2\varphi}\\
	\label{EQ:LIP8}&\leq C\left(\norml{2}{u-\tu}^2 + \norml{2}{\tu}^2(\norml{2}{\varphi}^2 + \norml{2}{\laplace\varphi}^2)\right) + \frac{m}{8}\norml{2}{\laplace^2\varphi}^2.
	\end{align}
	Now, we observe that
	\begin{align*}
	\laplace(\psi'(\varphi_u)-\psi'(\varphi_\tu)) &= \psi''(\varphi_u)\laplace\varphi + \laplace\varphi_\tu (\psi''(\varphi_u)-\psi''(\varphi_\tu))\\
	&\quad + \psi'''(\varphi_u)\left(\grad\varphi_u + \grad\varphi_\tu\right)\left(\grad\varphi_u-\grad \varphi_\tu \right) + \left(\psi'''(\varphi_u)-\psi'''(\varphi_\tu)\right)|\grad\varphi_\tu|^2.
	\end{align*}
	Due to the assumptions on $\psi(\cdot)$ and because of Proposition \ref{THM:EXS}, it is straightforward to check that
	\begin{equation*}
	\intO |\psi''(\varphi_u)\laplace\varphi|^2 + |\laplace\varphi_\tu (\psi''(\varphi_u)-\psi''(\varphi_\tu))|^2\dx \leq C\left(\norml{2}{\varphi}^2 + \norml{2}{\laplace\varphi}^2\right),
	\end{equation*}
	where we used the continuous embedding $H^2\subset L^{\infty}$ and elliptic regularity theory.
	With similar argument, using the Sobolev embedding $\H^1\subset\L^6$ and the assumptions on $\psi(\cdot)$, we obtain
	\begin{equation*}
	\intO |\psi'''(\varphi_u)\left(\grad\varphi_u + \grad\varphi_\tu\right)\left(\grad\varphi_u-\grad \varphi_\tu \right)|^2 + |\left(\psi'''(\varphi_u)-\psi'''(\varphi_\tu)\right)|^2|\grad\varphi_\tu|^4\dx\leq C\left(\norml{2}{\varphi}^2 + \norml{2}{\laplace\varphi}^2\right).
	\end{equation*}
	From the last two inequalities, we obtain
	\begin{equation}
		\label{EQ:LIP9}\norml{2}{\laplace(\psi'(\varphi_u)-\psi'(\varphi_\tu))}^2\leq C\left(\norml{2}{\varphi}^2 + \norml{2}{\laplace\varphi}^2\right).
	\end{equation}
	Therefore, we have
	\begin{equation}
	\label{EQ:LIP10}\left| m\intO \laplace\big((\psi'(\varphi_u)-\psi'(\varphi_\tu))\big)\laplace^2\varphi\dx\right| \leq C\left(\norml{2}{\varphi}^2 + \norml{2}{\laplace\varphi}^2\right)  + \frac{m}{8}\norml{2}{\laplace^2\varphi}^2.
	\end{equation}
	Plugging in \eqref{EQ:LIP7}-\eqref{EQ:LIP10} into \eqref{EQ:LIP6}, we obtain
	\begin{align}
	\nonumber
	\label{EQ:LIP10a}\frac{\d}{\dt}\intO |\laplace\varphi|^2\dx + m\intO |\laplace^2\varphi|^2\dx &\leq \left(\norml{2}{\varphi}^2 + \norml{2}{\tu}^2\norml{2}{\varphi}^2+\normH{2}{\v_\tu}^2\normh{1}{\varphi}^2+ \normh{2}{\sigma}^2 + \normH{1}{\v}^2  + \norml{2}{u-\tu}^2\right) \\
	&\qquad +C\left(1 + \norml{2}{\tu}^2\right) \norml{2}{\laplace\varphi}^2.
	\end{align}
	Invoking Proposition \ref{THM:EXS} and \eqref{EQ:LIP4} and using elliptic regularity theory, a Gronwall argument yields
	\begin{align}
	&\label{EQ:LIP11}\norm{\varphi}_{H^1((H^1)^*)\cap L^{\infty}(H^2)\cap L^2(H^3)} +  \norm{\laplace\varphi}_{L^2(H^2)} + \norm{\mu}_{L^2(H^1)} + \norm{\sigma}_{L^{\infty}(H^2)} + \norm{\v}_{L^2(\H^1)} + \norm{p}_{L^2(L^2)} \notag\\
	&\quad\leq C\norm{u-\tu}_{L^2(L^2)}. 
	\end{align}
	Using \eqref{EQ:LIP9} and 	\eqref{EQ:LIP11}, a comparison argument in \eqref{EQ:DIFF4} implies
	\begin{equation}
	\label{EQ:LIP12}\norm{\mu}_{L^{\infty}(L^2)\cap L^2(H^2)}\leq C\norm{u-\tu}_{L^2(L^2)}.
	\end{equation}
	Again using elliptic theory, from \eqref{EQ:DIFF4}, \eqref{EQ:DIFF6} and \eqref{EQ:LIP12} we obtain
	\begin{equation}
	\label{EQ:LIP13}\norm{\varphi}_{ L^2(H^4)}\leq C\norm{u-\tu}_{L^2(L^2)}.
	\end{equation}
	A further comparison argument in \eqref{EQ:DIFF3} together with \eqref{EQ:LIP12} yields
	\begin{equation}
	\label{EQ:LIP14}\norm{\delt \varphi}_{ L^2(L^2)}\leq C\norm{u-\tu}_{L^2(L^2)}.
	\end{equation}
	Summarising \eqref{EQ:LIP11}-\eqref{EQ:LIP14}, we obtain
	\begin{equation}
	\label{EQ:LIP15}\norm{\varphi}_{H^1(L^2)\cap L^{\infty}(H^2)\cap L^2(H^4)}   + \norm{\mu}_{L^{\infty}(L^2)\cap L^2(H^2)} + \norm{\sigma}_{L^{\infty}(H^2)} + \norm{\v}_{L^2(\H^1)} + \norm{p}_{L^2(L^2)}\leq C\norm{u-\tu}_{L^2(L^2)}. 
	\end{equation}
	Together with the assumptions on $\h(\cdot)$ and Gagliardo-Nirenberg's inequality, it follows that
	\begin{align}
	\label{EQ:LIP16}\norm{\divergence(\v)}_{L^8(H^1)}& \leq C\norm{u-\tu}_{L^2(L^2)},\\
	\label{EQ:LIP17} \norm{(\mu+\chi\sigma)\grad\varphi_u + (\mu_\tu+\chi\sigma_\tu)\grad\varphi}_{L^8(\L^2)}&\leq 	C\norm{u-\tu}_{L^2(L^2)}.
	\end{align}
	Then, an application of \cite[Lemma 1.5]{EbenbeckGarcke} yields
	\begin{equation}
	\norm{\v}_{L^8(\H^2)} + \norm{p}_{L^8(H^1)}\leq C\norm{u-\tu}_{L^2(L^2)}.
	\end{equation}
	Together with \eqref{EQ:LIP15}, this implies that
	\begin{equation}
	\label{EQ:LIP18}
	\norm{(\varphi,\mu,\sigma,\v, p)}_{\V_1}
	\leq C\norm{u-\tu}_{L^2(L^2)}. 
	\end{equation}
	Hence, setting $u=\tu$ completes the proof. 
\end{proof}	

Due to Proposition \ref{THM:EXS} and Theorem \ref{THM:UNS}, we can define an operator that maps any control $u\in\UR$ onto its corresponding state:
\begin{definition}
	\label{DEF:CSO}
	For any $u\in\UR$ we write $(\varphi_u,\mu_u,\sigma_u,\v_u,p_u)$ to denote the corresponding unique strong solution of \eqref{EQ:CHB} given by Proposition \ref{THM:EXS}.
	Then the operator
	\begin{align*}
	\S:\;\UR \to \V_1,\quad u\mapsto \S(u):=(\varphi_u,\mu_u,\sigma_u,\v_u,p_u)
	\end{align*}
	is called the \textbf{control-to-state} operator. 
\end{definition}

\begin{remark}
	The control-to-state operator is defined not only for admissible controls but for all controls in $\UR$. This will be especially important in subsection 3.4 because Fréchet differentiability is merely defined for open subsets of $L^2(L^2)$. Unlike the open ball $\UR$, the set $\U$ is closed and its interior is empty. Therefore it makes sense to investigate the control-to-state operator on the open superset $\UR$ instead. 
\end{remark}

In the following subsections, some properties of the control-to-state operator will be established that are essential for the treatment of optimal control problems.

\subsection{Lipschitz continuity}

The proof of Theorem \ref{THM:UNS} does actually provide more than uniqueness of strong solutions of \eqref{EQ:CHB}. In fact, we have showed that the strong solution depends Lipschitz-continuously on the control.

\begin{corollary}
	\label{COR:LIP}
	The \textbf{control-to-state} operator $S\colon \UR\to \V_1$ is Lipschitz continuous, i.e., there exists a constant $L_1>0$ depending only on the system parameters and on $R$, $\Omega$, $\Gamma$ and $T$ such that for all $u,\tu\in\UR$:
			\begin{align}
			\label{LEM:LIP:EST}\norm{\S(u)-\S(\tu)}_{\V_1}\leq L_1\norm{u-\tu}_{L^2(L^2)}. 
			\end{align}
\end{corollary}
\begin{proof}
The assertion follows directly from \eqref{EQ:LIP18}.
\end{proof}
\subsection{A weak compactness property}

As the control-to-state operator is nonlinear, the following result will be essential to prove existence of an optimal control (see Section 5.1):
\begin{lemma}
	\label{LEM:COMP}
	Suppose that $(u_k)_{k\in\N}\subset\U$ is converging weakly in $L^2(L^2)$ to some limit $\u\in\U$. Then
		\begin{alignat*}{6}
		&\varphi_{u_k} &&\wto \varphi_\u\quad &&\text{in}\; H^1(L^2)\cap L^2(H^4),\quad
		&&\varphi_{u_k} &&\to \varphi_\u &&\text{in}\; C\big([0,T];W^{1,r}\big)\cap C\big(\overline{\OmegaT}\big),~r\in [1,6),\\
		&\mu_{u_k}&&\wto\mu_\u\quad &&\text{in}\; L^2(H^2),
		&&\v_{u_k}&&\wto\v_\u\quad &&\text{in}\; L^2(\H^2),\\
		&\sigma_{u_k}&&\wto\sigma_\u\quad &&\text{in}\; L^2(H^2), \quad  
		&&p_{u_k}&&\wto p_\u\quad &&\text{in}\; L^2(H^1),
		\end{alignat*}
after extraction of a subsequence, where the limit $(\varphi_\u,\mu_\u,\sigma_\u,\v_\u,p_\u)$ is the strong solution of \eqref{EQ:CHB} to the control $\u \in\U$.
\end{lemma}

\begin{proof} 
	The assertion follows with exactly the same arguments as the proof of \cite[Lem.\,8]{EbenbeckKnopf}.
\end{proof}	

\begin{remark}
	This result actually means weak compactness of the control-to-state operator restricted to $\U$ since any bounded sequence in $\U$ has a weakly convergent subsequence according to the Banach-Alaoglu theorem. However, this property can not be considered as weak continuity as the extraction of a subsequence is necessary.
\end{remark}
	
\subsection{The linearized system}  

We want to show that the control-to-state operator is also Fréchet differentiable on the open ball $\UR$ (and therefore especially on its strict subset $\U$). Since the Fréchet derivative is a linear approximation of the control-to-state operator at some certain point $u\in\UR$, it will be given by a linearized version of \eqref{EQ:CHB}:
\begin{subequations}
	\label{EQ:LIN}
\begin{empheq}[left=\textnormal{(LIN)}\empheqlbrace]{align}
\label{EQ:LIN1}
\divergence(\v) &= P\sigma\h(\varphi_u) + (P\sigma_u-A)\h'(\varphi_u)\varphi + F_1&&\text{in}\; \OmegaT,\\
\label{EQ:LIN2}
-\divergence(\T(\v,p)) + \nu\v &= (\mu_u+\chi\sigma_u)\grad\varphi + (\mu+\chi\sigma)\grad\varphi_u + \F &&\text{in}\; \OmegaT,\\
\label{EQ:LIN3}
\delt\varphi + \divergence(\varphi_u\v) + \divergence(\varphi\v_u) &= m\laplace\mu + (P\sigma_u - A - u)\h'(\varphi_u)\varphi \notag\\
&\qquad + P\sigma\h(\varphi_u) + F_2&&\text{in}\; \OmegaT,\\
\label{EQ:LIN4}
\mu &= -\laplace\varphi + \psi''(\varphi_u)\varphi - \chi\sigma + F_3 &&\text{in}\; \OmegaT,\\
\label{EQ:LIN5}
-\laplace\sigma - b\sigma + \h'(\varphi_u)\varphi\sigma_u + \h(\varphi_u)\sigma &=  F_4 &&\text{in}\; \OmegaT,\\[2mm]
\label{EQ:LIN6}
\deln\mu=\deln\varphi = \deln\sigma &= 0 &&\text{in}\; \GammaT,\\
\label{EQ:LIN7}
\T(\v,p)\n &= 0 &&\text{in}\; \GammaT,\\[2mm]
\label{EQ:LIN9}
\varphi(0) &= 0	&&\text{in}\; \Omega,
\end{empheq}
\end{subequations}
where $F_i:\OmegaT\to\R,\,1\le i\le 4$ and $\F:\OmegaT\to\R^3$ are given functions that will be specified later on. A strong solution of this linearized system is defined as follows:
\begin{definition}
	Let $u\in\UR$ be arbitrary. Then a quintuplet $(\varphi,\mu,\sigma,\v,p)$ is called a strong solution of \eqref{EQ:LIN} if
	it lies in $\V_1$ and satisfies \eqref{EQ:LIN} almost everywhere in the respective sets.
\end{definition}
Existence and uniqueness of strong solutions to this linearized system is established by the following lemma:
\begin{proposition}
	\label{PROP:LIN}
	Let $u\in\UR$ be any control and let $(\varphi,\mu,\sigma,\v,p)$ denote its corresponding state. Moreover, let $(\F,F_1,F_2,F_3,F_4)\in \V_2$
	be arbitrary. Then the system \eqref{EQ:LIN}
	has a unique strong solution $(\varphi,\mu,\sigma,\v,p)\in\V_1$. Moreover, there exists some constant $C>0$ depending only on the system parameters and on $R$, $\Omega$, $\Gamma$ and $T$ such that:
	\begin{align}
	\label{EST:LIN}
		\norm{(\varphi,\mu,\sigma,\v,p)}_{\V_1} \le C\norm{(\F,F_1,F_2,F_3,F_4)}_{\V_2}.
	\end{align}
\end{proposition}

\begin{proof}
The proof can be carried out using a Galerkin scheme by constructing approximate solutions with respect to $\varphi$ and $\mu$ and at the same time solve for $\sigma,\,\v$ and $p$ in the corresponding whole function spaces. For the details, we refer to \cite[Sec. 3.5]{EbenbeckKnopf}
In the following steps, we will show a-priori-estimates for the solutions. All the estimates can be carried out rigorously within the Galerkin scheme. In the following approach, Hölder's and Young's inequalities will be used frequently. 
\paragraph{Step 1:} Multiplying \eqref{EQ:LIN5} with $\sigma$, integrating by parts and using \eqref{EQ:LIN6}, the boundedness of $\h'(\varphi_u)\in L^{\infty}(\OmegaT),\,\sigma_u\in L^{\infty}(L^{\infty})$ and the non-negativity of $\h(\cdot)$ yields
\begin{equation*}
\normL{2}{\grad\sigma}^2 + b\norml{2}{\sigma}^2 = \left|\intO \big(F_4-\h'(\varphi_u)\varphi\sigma_u\big)\sigma\dx\right| \leq\delta_1\normh{1}{\sigma}^2 + C_{\delta_1}\left(\norml{2}{\varphi}^2 + \norml{2}{F_4}^2\right)
\end{equation*}
for $\delta_1>0$ arbitrary. Choosing $\delta_8 = \frac{1}{2}\min\{1,b\}$, this implies that
\begin{equation}
\label{lin_eq_1}\normh{1}{\sigma}\leq C\left(\norml{2}{\varphi} + \norml{2}{F_4}\right).
\end{equation}
Then, using exactly the same arguments as in \cite{EbenbeckKnopf}, it can be shown that
\begin{align}
\nonumber&\norm{\varphi}_{L^{\infty}(H^1)\cap L^2(H^3)} + \norm{\mu}_{L^2(H^1)} + \norm{\sigma}_{L^2(H^1)}+ \norm{\v}_{L^2(\H^1)} + \norm{p}_{L^2(L^2)}\\
\label{lin_eq_2} &\quad \leq C\left(\norm{\grad F_3}_{L^2(\L^2)} + \norm{\F}_{L^2(\L^2)} + \sum_{i=1}^{4}\norm{F_i}_{L^2(\OmegaT)}\right).
\end{align}
\paragraph{Step 2:} We want to establish higher order estimates for $\varphi$, $\mu$ and $\sigma$. With \eqref{EQ:LIN5}-\eqref{EQ:LIN6} and elliptic regularity theory, it follows that
\begin{equation*}
\normh{2}{\sigma}\leq C\left(\norml{2}{\h'(\varphi_u)\varphi\sigma_u} + \norml{2}{F_4}\right).
\end{equation*}
Due to the assumptions on $\h(\cdot)$, using Proposition \ref{THM:EXS} and \eqref{lin_eq_2} implies
\begin{equation}
\label{lin_eq_3}\norm{\sigma}_{L^{\infty}(H^2)}\leq C\left(\norm{\grad F_3}_{L^2(\L^2)} + \norm{\F}_{L^2(\L^2)} + \sum_{i=1}^{3}\norm{F_i}_{L^2(\OmegaT)} + \norm{F_4}_{L^{\infty}(L^2)}\right).
\end{equation}
We now multiply \eqref{EQ:LIN3} with $\laplace^2\varphi$, integrate by parts and insert the expression for $\mu$ given by \eqref{EQ:LIN4} to obtain
\begin{align}
\nonumber \frac{\d}{\dt}\frac{1}{2}\intO |\laplace\varphi|^2\dx + m\intO |\laplace^2\varphi|^2\dx &= \intO \big((P\sigma_u - A - u)\h'(\varphi_u)\varphi + P\sigma\h(\varphi_u) + F_2\big)\laplace^2\varphi\dx\\
\nonumber &\quad - \intO \big(\divergence(\varphi_u\v) + \divergence(\varphi\v_u)\big)\laplace^2\varphi\dx \\
\label{lin_eq_4}&\quad + m\intO \laplace \big(\psi''(\varphi_u)\varphi - \chi\sigma + F_3\big)\laplace^2\varphi\dx.
\end{align}
Using Proposition \ref{THM:EXS}, the assumptions on $\h(\cdot)$, \eqref{lin_eq_3}-\eqref{lin_eq_4} and the continuous embeddings $H^1\subset L^6,\,\H^1\subset \L^6,\,\H^2\subset \L^{\infty}$, we have
\begin{align}
\nonumber &\left|\intO \big((P\sigma_u - A)\h'(\varphi_u)\varphi + P\sigma\h(\varphi_u) + F_2- \divergence(\varphi_u\v) + \divergence(\varphi\v_u)\big)\laplace^2\varphi\dx\right| \\
\label{lin_eq_5}&\quad \leq C\left(1+\normH{2}{\v_u}^2\right)\normh{1}{\varphi}^2 + C\left(\norml{2}{\sigma}^2 + \normH{1}{\v}^2 + \norml{2}{F_2}^2\right) + \frac{m}{8}\norml{2}{\laplace^2\varphi}^2.
\end{align}
Furthermore, it is straightforward to check that
\begin{equation}
\label{lin_eq_6}\left| m\intO \laplace \big(- \chi\sigma + F_3\big)\laplace^2\varphi\dx\right| \leq C\left(\normh{2}{\sigma}^2 + \normh{2}{F_3}^2\right) + \frac{m}{8}\norml{2}{\laplace^2\varphi}^2.
\end{equation}
Now, using elliptic regularity theory, the Sobolev embedding $H^2\subset L^{\infty}$, the assumptions on $\h(\cdot)$ and Proposition \ref{THM:EXS}, we calculate
\begin{equation}
\label{lin_eq_7}\left|\intO u\h'(\varphi_u)\varphi\laplace^2\varphi\dx \right| \leq C\norml{2}{u}\normh{2}{\varphi}\norml{2}{\laplace^2\varphi}\leq C\norml{2}{u}^2\left(\norml{2}{\varphi}^2 + \norml{2}{\laplace\varphi}^2\right)+ \frac{m}{8}\norml{2}{\laplace^2\varphi}^2.
\end{equation}
Next, we observe that
\begin{equation*}
\laplace(\psi''(\varphi_u)\varphi) = \psi^{(4)}(\varphi_u)|\grad\varphi_u|^2\varphi + \psi'''(\varphi_u)\laplace\varphi_u\varphi + 2\psi'''(\varphi_u)\grad\varphi_u\cdot\grad\varphi + \psi''(\varphi_u)\laplace\varphi.
\end{equation*}
Using the Sobolev embeddings $H^1\subset L^6,\,H^2\subset L^{\infty},\,\H^1\subset \L^6$, the assumptions on $\psi(\cdot)$, Proposition \ref{THM:EXS} and elliptic regularity theory again, we obtain
\begin{equation}
\label{lin_eq_8}\norml{2}{\laplace(\psi''(\varphi_u)\varphi)}^2\leq C\left(\norml{2}{\varphi}^2 + \norml{2}{\laplace\varphi}^2\right).
\end{equation}
Consequently, 
\begin{equation}
\label{lin_eq_9}\left|m\intO \laplace(\psi''(\varphi_u)\varphi)\laplace^2\varphi\dx \right|\leq C\norml{2}{\laplace(\psi''(\varphi_u)\varphi)}^2 + \frac{m}{8}\norml{2}{\laplace^2\varphi}^2\leq C\left(\norml{2}{\varphi}^2 + \norml{2}{\laplace\varphi}^2\right) +\frac{m}{8}\norml{2}{\laplace^2\varphi}^2.
\end{equation}
Plugging in \eqref{lin_eq_5}-\eqref{lin_eq_9} into \eqref{lin_eq_4}, we obtain
\begin{align*}
 \frac{\d}{\dt}\frac{1}{2}\intO |\laplace\varphi|^2\dx + \frac{m}{2}\intO |\laplace^2\varphi|^2\dx &\leq C\left(1+\normH{2}{\v_u}^2\right)\normh{1}{\varphi}^2 + C\left(\normh{2}{\sigma}^2 + \normH{1}{\v}^2 + \norml{2}{F_2}^2 + \normh{2}{F_3}^2\right) \\
 &\quad + C\left(1+\norml{2}{u}^2\right)\left(\norml{2}{\varphi}^2 + \norml{2}{\laplace\varphi}^2\right).
\end{align*}
Integrating this inequality in time from $0$ to $T$, using \eqref{lin_eq_2}-\eqref{lin_eq_3}, Proposition \ref{THM:EXS} and elliptic regularity theory, we end up with
\begin{align}
\nonumber&\norm{\varphi}_{L^{\infty}(H^2)\cap L^2(H^3)} + \norm{\laplace^2\varphi}_{L^2(L^2)} + \norm{\mu}_{L^2(H^1)} + \norm{\sigma}_{L^{\infty}(H^2)}+ \norm{\v}_{L^2(\H^1)} + \norm{p}_{L^2(L^2)}\\
\label{lin_eq_10} &\quad \leq C\left(\norm{\F}_{L^2(\L^2)} + \sum_{i=1}^{2}\norm{F_i}_{L^2(\OmegaT)} + \norm{F_3}_{L^2(H^2)}+ \norm{F_4}_{L^{\infty}(L^2)}\right).
\end{align}
Now, using elliptic regularity theory and \eqref{EQ:LIN4}, \eqref{EQ:LIN6}, \eqref{lin_eq_8} and \eqref{lin_eq_10}, we deduce that
\begin{equation}
\label{lin_eq_11}\norm{\mu}_{L^2(H^2)}\leq C\left(\norm{\F}_{L^2(\L^2)} + \sum_{i=1}^{2}\norm{F_i}_{L^2(\OmegaT)} + \norm{F_3}_{L^2(H^2)}+ \norm{F_4}_{L^{\infty}(L^2)}\right).
\end{equation}
Furthermore, using Proposition \ref{THM:EXS}, the assumptions on $\psi(\cdot)$ and \eqref{lin_eq_10}, a comparison argument in \eqref{EQ:LIN4} yields
\begin{equation}
\label{lin_eq_12}\norm{\mu}_{L^{\infty}(L^2)}\leq C\, C_F
\end{equation}
where
\begin{equation}
	C_F \coloneqq \left(\norm{\F}_{L^2(\L^2)} + \sum_{i=1}^{2}\norm{F_i}_{L^2(\OmegaT)} + \norm{F_3}_{L^{\infty}(L^2)\cap L^2(H^2)}+ \norm{F_4}_{L^{\infty}(L^2)}\right).
\end{equation}
Now, using elliptic regularity, Proposition \ref{THM:EXS}, the assumptions on $\psi(\cdot)$ and \eqref{lin_eq_8}, \eqref{lin_eq_10}, one can check that
\begin{equation*}
\norm{\psi''(\varphi_u)\varphi}_{L^2(H^2)} \leq C\, C_F.
\end{equation*}
Hence, using elliptic regularity theory again and recalling \eqref{lin_eq_10}-\eqref{lin_eq_11}, from \eqref{EQ:LIN4} we deduce 
\begin{equation}
\label{lin_eq_13}\norm{\varphi}_{L^2(H^4)}\leq C\, C_F.
\end{equation}
A further comparison in \eqref{EQ:LIN3} together with \eqref{lin_eq_10}-\eqref{lin_eq_13} yields
\begin{equation}
\label{lin_eq_14}\norm{\delt\varphi}_{L^2(L^2)}\leq C\, C_F.
\end{equation}
Summarising \eqref{lin_eq_10}-\eqref{lin_eq_13}, we showed that
\begin{align}
\label{lin_eq_15}
\norm{\varphi}_{H^1(L^2)\cap L^{\infty}(H^2)\cap L^2(H^4)} + \norm{\mu}_{L^{\infty}(L^2)\cap L^2(H^2)} + \norm{\sigma}_{L^{\infty}(H^2)}+ \norm{\v}_{L^2(\H^1)} + \norm{p}_{L^2(L^2)} 
	\le C\, C_F.
\end{align}
\paragraph{Step 3:} Now, we also want to prove higher order estimates for $\v$ and $p$. Using Proposition \ref{THM:EXS}, the assumptions on $\h(\cdot)$ and \eqref{lin_eq_15}, a straightforward calculation shows that
\begin{align}
\label{lin_eq_16}
\norm{P\sigma\h(\varphi_u) + (P\sigma_u-A)\h'(\varphi_u)\varphi}_{L^8(H^1)} \le C\, C_F.
\end{align}
Using Gagliardo-Nirenberg's inequality, we have the continuous embedding
\begin{equation*}
L^{\infty}(\H^1)\cap L^2(\H^3)\hookrightarrow L^8(\L^{\infty}).
\end{equation*}
Together with Proposition \ref{THM:EXS} and \eqref{lin_eq_15}, this implies that
\begin{align}
\label{lin_eq_17}
\norm{(\mu_u+\chi\sigma_u)\grad\varphi + (\mu+\chi\sigma)\grad\varphi_u}_{L^8(\L^2)}\le C \, C_F.
\end{align}
Using \eqref{lin_eq_16}-\eqref{lin_eq_17} and recalling \eqref{lin_eq_15}, an application of \cite[Lemma 1.5]{EbenbeckGarcke} to \eqref{EQ:LIN1}-\eqref{EQ:LIN2}, \eqref{EQ:LIN7} yields
\begin{align}
\label{lin_eq_18}
\norm{(\varphi,\mu,\sigma,\v,p)}_{\V_1} \le C\norm{(\F,F_1,F_2,F_3,F_4)}_{\V_2},
\end{align}
hence we showed \eqref{EST:LIN}.
\paragraph{Step 4:} Due to (\ref{lin_eq_18}), we can pass to the limit in the Galerkin scheme to deduce that \eqref{EQ:LIN} holds. The initial condition is attained due to the compact embedding $H^1((H^1)^*)\cap L^{\infty}(H^1)\hookrightarrow C([0,T];L^2)$ (see \cite[sect.\,8, Cor.\,4]{Simon}). Moreover, the estimate \eqref{EST:LIN} results from the weak-star lower semicontinuity of the $\V_1$-norm. Finally, uniqueness follows from linearity of the system together with (\ref{EST:LIN}).
\end{proof}

\subsection{Fréchet differentiability}

Now, this result can be used to prove Fréchet differentiability of the control-to-state operator:
\begin{proposition}
	\label{LEM:FRE} The following statements hold:
	\begin{enumerate}
	\itemi
	The control-to-state operator $S$ is Fréchet differentiable on $\UR$, i.e., for any $u\in\UR$ there exists a unique bounded linear operator 
	\begin{align*}
		\S'(u): L^2(L^2) \to \V_1, \quad h \mapsto \S'(u)[h]=\big(\varphi'_u,\mu'_u,\v'_u,\sigma'_u,p'_u\big)[h],
	\end{align*}
	such that
	\begin{align*}
		\frac{\norm{\S(u+h)-\S(u)-\S'(u)[h]}_{\V_1}}{\norm{h}_{L^2(L^2)}} \to 0 \qquad\text{as}\; \norm{h}_{L^2(L^2)}\to 0.
	\end{align*}
	For any $u\in\U$ and $h\in L^2(L^2)$, the Fréchet derivative $\big(\varphi'_u,\mu'_u,\v'_u,\sigma'_u,p'_u\big)[h]$ is the unique strong solution of the system \eqref{EQ:LIN} with 
	\begin{align*}
		F_1,F_3,F_4=0,\quad \F=\mathbf{0} \tand F_2= -h\,\h(\varphi_u).
	\end{align*}
	\itemii The Frechet-derivative is Lipschitz continuous, i.e., for any $u,\tu\in \UR$ and $h\in L^2(L^2)$, it holds that
	\begin{equation}
	\label{LIP:EST:FRE:DER} \norm{\S'(u)[\cdot]-\S'(\tu)[\cdot]}_{\mathcal L(L^2(L^2);\V_1)} \leq L_2 \norm{u-\tu}_{L^2(L^2)},
	\end{equation}
	with a constant $L_2>0$ depending only on the system parameters and on $\Omega$, $T$, $R$.
\end{enumerate}
\end{proposition}

\begin{proof}
	Let $C$ denote a generic nonnegative constant that depends only on $R$, $\Omega$ and $T$ and may change its value from line to line. 
	\paragraph{Proof of \textnormal{(i)}:}
	 To prove Fréchet differentiability we must consider the difference 
	\begin{align*}
		(\varphi,\mu,\sigma,\v,p):=(\varphi_{u+h},\mu_{u+h},\v_{u+h},\sigma_{u+h},p_{u+h})-(\varphi_u,\mu_u,\sigma_u,\v_u,p_u)
	\end{align*}
	for some arbitrary $u\in\UR$ and $h\in L^2(L^2)$ with $u+h\in\UR$. Therefore, we assume that $\norm{h}_{L^2(L^2)}<\delta$ for some sufficiently small $\delta>0$. Now, we Taylor expand the nonlinear terms in \eqref{EQ:CHB} to pick out the linear contributions. We obtain that
	\begin{align*}
		\h(\varphi_{u+h}) - \h(\varphi_{u}) &= \h'(\varphi_u)\varphi + \CR_1, \\
		\sigma_{u+h}\h(\varphi_{u+h}) - \sigma_{u}\h(\varphi_{u}) &= \sigma\h(\varphi_{u}) + \sigma_{u}\h'(\varphi_u)\varphi + \CR_2, \\
		(u+h)\h(\varphi_{u+h}) - u\h(\varphi_{u}) &= u\h'(\varphi_u)\varphi + h\h(\varphi_u) + \CR_3,\\
		(\mu_{u+h}+\chi\sigma_{u+h})\grad\varphi_{u+h} - (\mu_{u}+\chi\sigma_{u})\grad\varphi_{u} &= (\mu_{u}+\chi\sigma_{u})\grad\varphi + (\mu+\chi\sigma)\grad\varphi_{u} + \CR_4, \\
		\divergence(\varphi_{u+h}\v_{u+h}) - \divergence(\varphi_{u}\v_{u}) &= \divergence(\varphi\v_{u}) + \divergence(\varphi_{u}\v) +\CR_5,\\
		\psi'(\varphi_{u+h}) - \psi'(\varphi_{u}) &= \psi''(\varphi_{u})\varphi + \CR_6,
	\end{align*}
	where the nonlinear remainders are given by
	\begin{align*}
		\CR_1&:= \tfrac 1 2 \h''(\zeta)(\varphi_{u+h} - \varphi_{u})^2,	\\
		\CR_2&:= (\sigma_{u+h}-\sigma_u)(\h(\varphi_{u+h}) - \h(\varphi_{u})) + \tfrac 1 2 \sigma_u\h''(\zeta)(\varphi_{u+h} - \varphi_{u})^2,\\
		\CR_3&:= \tfrac 1 2 u\,\h''(\zeta)(\varphi_{u+h} - \varphi_{u})^2 + h\big(\h(\varphi_{u+h})-\h(\varphi_{u})\big),\\
		\CR_4&:= \big[(\mu_{u+h} - \mu_{u}) +\chi(\sigma_{u+h}-\sigma_{u})\big]
		(\grad\varphi_{u+h} - \grad\varphi_{u}),\\	
		\CR_5&:= \divergence\big[(\varphi_{u+h} - \varphi_{u})(\v_{u+h} - \v_{u})\big],\\
		\CR_6&:= \tfrac 1 2 \psi'''(\xi) (\varphi_{u+h} - \varphi_{u})^2	
	\end{align*}
	with $\zeta=\vartheta\varphi_{u+h}+(1-\vartheta)\varphi_{u}$ and $\xi=\theta\varphi_{u+h}+(1-\theta)\varphi_{u}$ for some $\vartheta,\theta\in[0,1]$. This means that the difference $(\varphi,\mu,\sigma,\v,p)$ is the strong solution of \eqref{EQ:LIN} with
	\begin{gather*}
		F_1 = P\CR_2 - A\CR_1,\quad \F = \CR_4,\quad F_2 =  P\CR_2 - A\CR_1 - \CR_3-\CR_5 -h\,\h(\varphi_u),\quad F_3 = \CR_6,\quad F_4=-\CR_2.
	\end{gather*}
	By a simple computation, one can show that these functions have the desired regularity. Now, we write $(\varphi_u^h,\mu_u^h,\v_u^h,\sigma_u^h,p_u^h)$ to denote the strong solution of \eqref{EQ:LIN} with 
	\begin{align*}
		F_1,F_3,F_4=0,\quad \F=\mathbf{0} \tand F_2= -h\,\h(\varphi_u).
	\end{align*}
	and $(\varphi_{\CR}^h,\mu_{\CR}^h,\v_{\CR}^h,\sigma_{\CR}^h,p_{\CR}^h)$ to denote the strong solution of \eqref{EQ:LIN} with 
	\begin{align}
	\label{EQ:F}
		F_1 = P\CR_2 - A\CR_1,\;\; \F = \CR_4,\;\;\, F_2 =  P\CR_2 - A\CR_1 - \CR_3-\CR_5,\;\;\, F_3 = \CR_6,\;\;\, F_4=-\CR_2.
	\end{align}
	Because of linearity of the system \eqref{EQ:LIN} and uniqueness of its solution, it follows that
	\begin{align*}
		(\varphi_{u+h},\mu_{u+h},\v_{u+h},\sigma_{u+h},p_{u+h}) - (\varphi_u,\mu_u,\sigma_u,\v_u,p_u) - (\varphi_u^h,\mu_u^h,\v_u^h,\sigma_u^h,p_u^h)
			=(\varphi_{\CR}^h,\mu_{\CR}^h,\v_{\CR}^h,\sigma_{\CR}^h,p_{\CR}^h).
	\end{align*}
	We conclude from Proposition \ref{THM:EXS} that $\zeta$ and $\xi$ are uniformly bounded. This yields
	\begin{align*}
		\norm{\psi^{(i)}(\zeta)}_{L^\infty(\OmegaT)}\le C\quad\forall\, 1\leq i\leq 4,\quad\tand \norm{\h^{(j)}(\zeta)}_{L^\infty(\OmegaT)}\le C\quad \forall\, 1\leq j\leq 3.
	\end{align*} 
	Moreover, since $\h(\cdot)$ is Lipschitz continuous, it holds that
	\begin{align*}
		\norm{\h(\varphi_{u+h})-\h(\varphi_{u})}_{L^\infty(\OmegaT)} \le C\,\norm{\varphi_{u+h}-\varphi_{u}}_{L^\infty(\OmegaT)} \le C\,\norm{h}_{L^2(L^2)}.
	\end{align*} 
	Together with the Lipschitz estimates from Corollary \ref{COR:LIP} we obtain that
	\begin{align*}
		\norm{\CR_i}_{L^2(L^2)} &\le C\, \norm{h}_{L^2(L^2)}^{2},\quad i\in\{1,2,3,6\},\\
		\norm{\CR_i}_{L^{\infty}(L^2)}&\le C\, \norm{h}_{L^2(L^2)}^{2},\quad i\in\{1,2,6\}.
	\end{align*} 
	Moreover, we have
	\begin{align}
	\label{EST:R51}
		\norm{(\grad\varphi_{u+h}-\grad\varphi_{u})(\v_{u+h}-\v_{u})}_{L^2(L^2)} \notag& \le \norm{\grad\varphi_{u+h}-\grad\varphi_{u}}_{L^\infty(L^3)}\, \norm{\v_{u+h}-\v_{u}}_{L^2(L^6)} \notag\\
		& \le C\, \norm{\varphi_{u+h}-\varphi_{u}}_{L^\infty(H^1)}\, \norm{\v_{u+h}-\v_{u}}_{L^2(H^1)}
	\end{align} 
	and then Corollary \ref{COR:LIP} yields
	\begin{align*}
		\norm{\CR_5}_{L^2(L^2)} 
		&\le C\,\norm{\varphi_{u+h}-\varphi_{u}}_{L^\infty(\OmegaT)}\, \norm{\v_{u+h}-\v_{u}}_{L^2(H^1)} 
			+ C\,\norm{(\grad\varphi_{u+h}-\grad\varphi_{u})(\v_{u+h}-\v_{u})}_{L^2(L^2)} \notag\\
		&\le C\,\norm{h}_{L^2(L^2)}^2.
	\end{align*} 
	Due to the continuous embedding $L^{\infty}(\H^1)\cap L^2(\H^3)\hookrightarrow L^8(\L^{\infty})$ resulting from Gagliardo-Nirenberg's inequality, an application of Corollary \ref{COR:LIP} gives
	\begin{equation}
	\norm{\CR_4}_{L^8(\L^2)}\leq C\norm{h}_{L^2(L^2)}^2.
	\end{equation}
	Furthermore, we have
	\begin{align*}
		\norm{\grad\CR_6}_{L^2(L^2)} 
		&\le  C\, \norm{\grad\xi}_{L^\infty(L^6)} \, \norm{\varphi_{u+h}-\varphi_u}_{L^\infty(L^6)}^2 
			+ C\, \norm{\grad\varphi_{u+h}-\grad\varphi_u}_{L^2(L^3)} \, \norm{\varphi_{u+h}-\varphi_u}_{L^\infty(L^6)} \\
		&\le  C\, \norm{\grad\xi}_{L^\infty(\H^1)} \, \norm{\varphi_{u+h}-\varphi_u}_{L^\infty(H^1)}^2 
			+ C\, \norm{\varphi_{u+h}-\varphi_u}_{L^2(H^2)} \, \norm{\varphi_{u+h}-\varphi_u}_{L^\infty(H^1)} \\
		&\le C\, \norm{h}_{L^2(L^2)}^{2}
	\end{align*} 
	and
	\begin{align*}
	\norm{\laplace\CR_6}_{L^2(L^2)} 
	\le  C\, \left( 1+ \norm{\xi}_{L^\infty(H^2)}^2\right) \, \norm{\varphi_{u+h}-\varphi_u}_{L^\infty(H^2)}^2 \le C\, \norm{h}_{L^2(L^2)}^{2}.
	\end{align*} 
	From the last two inequalities and elliptic regularity theory, we infer that
	\begin{equation*}
	\norm{\CR_6}_{L^2(H^2)}  \le C\, \norm{h}_{L^2(L^2)}^{2}.
	\end{equation*}
	Now, we first observe that
	\begin{align*}
	\norm{\grad\CR_1}_{L^8(\L^2)} &\leq \norm{\h^{(3)}(\zeta)\grad\zeta\, (\varphi_{u+h} - \varphi_{u})^2}_{L^8(\L^2)} + \norm{2\h''(\zeta)\grad (\varphi_{u+h} - \varphi_{u})\,(\varphi_{u+h} - \varphi_{u})}_{L^8(\L^2)}\\
	&\leq C\norm{\grad \zeta}_{L^{\infty}(\L^6)}\norm{\varphi_{u+h}}_{L^{\infty}(L^6)}^2 + C\norm{\grad (\varphi_{u+h}-\varphi_{u})}_{L^{\infty}(\L^6)}\norm{\varphi_{u+h}-\varphi_{u}}_{L^{\infty}(L^3)}\\
	&\leq C\, \norm{h}_{L^2(L^2)}^{2}.
	\end{align*}
	With similar arguments, it follows that
	\begin{align*}
	\norm{\grad\big(\sigma_u\h''(\zeta)(\varphi_{u+h} - \varphi_{u})^2\big)}_{L^8(\L^2)}&\leq \norm{\grad\sigma_u\h''(\zeta)(\varphi_{u+h} - \varphi_{u})^2}_{L^8(\L^2)} + \norm{2\sigma_u\grad\CR_1}_{L^8(\L^2)} \\
	&\leq C\norm{\grad\sigma_u}_{L^{\infty}(L^6)}\norm{\varphi_{u+h}}_{L^{\infty}(L^6)}^2 + C\norm{\sigma_u}_{L^{\infty}(\OmegaT)}\norm{\grad\CR_1}_{L^8(\L^2)}\\
	&\leq C\, \norm{h}_{L^2(L^2)}^{2}.
	\end{align*}
	From the Lipschitz-continuity of $\h'(\cdot)$, we deduce that
	\begin{equation*}
	\norm{\grad\big((\sigma_{u+h}-\sigma_u)(\h(\varphi_{u+h}) - \h(\varphi_{u}))\big)}_{L^8(\L^2)}\leq C\norm{\sigma_{u+h}-\sigma_u}_{L^{\infty}(H^2)}\norm{\varphi_{u+h}-\varphi_u}_{L^{\infty}(H^2)}\leq C\, \norm{h}_{L^2(L^2)}^{2}.
	\end{equation*}
	The last two inequalities imply
	\begin{equation*}
	\norm{\CR_i}_{L^8(H^1)}\leq C\, \norm{h}_{L^2(L^2)}^{2},\quad i\in\{1,2\}.
	\end{equation*}
	This finally yields
	\begin{gather*}
		\norm{(\F,F_1,F_2,F_3,F_4)}_{\V_2}\le C\, \norm{h}_{L^2(L^2)}^{2},
	\end{gather*}
	where $F_i$ denote the functions given by \eqref{EQ:F}. Hence, due to \eqref{EST:LIN} we obtain that
	\begin{align*}
		\norm{(\varphi_{\CR}^h,\mu_{\CR}^h,\sigma_{\CR}^h,\v_{\CR}^h,p_{\CR}^h)}_{\V_1} \le C\, \norm{h}_{L^2(L^2)}^2,
	\end{align*}
	which completes the proof of (i).
	\paragraph{Proof of \textnormal{(ii)}:} 
	In the following, we write
	\begin{equation*}
	(\varphi,\mu,\sigma,\v,p)\coloneqq \big(\varphi'_u,\mu'_u,\sigma'_u,\v'_u,p'_u\big)[h]-\big(\varphi'_\tu,\mu'_\tu,\sigma'_\tu,\v'_\tu,p'_\tu\big)[h].
	\end{equation*}
	Then, using the mean value theorem, a long but straightforward calculation shows that
	\begin{subequations}
		\label{EQ:FRE:LIP}
	\begin{align}
	\label{EQ:FRE:LIP1}
	\divergence(\v) &= P\sigma\h(\varphi_u) +  (P\sigma_u-A)\h'(\varphi_u)\varphi + F_1&&\text{in}\; \OmegaT,\\
	\label{EQ:FRE:LIP2}
	-\divergence(\T(\v,p)) + \nu\v &= (\mu_u+\chi\sigma_u)\grad\varphi + (\mu+\chi\sigma)\grad\varphi_u + \F &&\text{in}\; \OmegaT,\\
	\label{EQ:FRE:LIP3}
	\delt\varphi + \divergence(\varphi_u\v) + \divergence(\varphi\v_u) &= m\laplace\mu + (P\sigma_u - A - u)\h'(\varphi_u)\varphi \notag\\
	&\qquad + P\sigma\h(\varphi_u) + F_2&&\text{in}\; \OmegaT,\\
	\label{EQ:FRE:LIP4}
	\mu &= -\laplace\varphi + \psi''(\varphi_u)\varphi - \chi\sigma + F_3 &&\text{in}\; \OmegaT,\\
	\label{EQ:FRE:LIP5}
	-\laplace\sigma - b\sigma + \h'(\varphi_u)\varphi\sigma_u + \h(\varphi_u)\sigma &=  F_4 &&\text{in}\; \OmegaT,\\[2mm]
	\label{EQ:FRE:LIP6}
	\deln\mu=\deln\varphi = \deln\sigma &= 0 &&\text{in}\; \GammaT,\\
	\label{EQ:FRE:LIP7}
	\T(\v,p)\n &= 0 &&\text{in}\; \GammaT,\\[2mm]
	\label{EQ:FRE:LIP9}
	\varphi(0) &= 0	&&\text{in}\; \Omega,
	\end{align}
	\end{subequations}
	where
	\begin{align*}
	F_1&\coloneqq P\sigma_\tu'[h](\h(\varphi_u)-\h(\varphi_\tu))+(P\sigma_\tu-A)\varphi_\tu'[h]\h'(\xi)(\varphi_u-\varphi_\tu) + P\varphi_\tu'[h](\sigma_u-\sigma_\tu)\h'(\varphi_u),\\
	\F&\coloneqq \varphi_\tu'[h]\big((\mu_u+\chi\sigma_u)-(\mu_\tu + \chi\sigma_\tu)\big) + (\mu_\tu'[h]+\chi\sigma_\tu'[h])(\grad\varphi_u - \grad\varphi_\tu),\\
	F_2&\coloneqq F_1 - \varphi_\tu'[h](u-\tu)\h'(\varphi_u) - \tu\h'(\xi)(\varphi_u-\varphi_\tu)-h(\h(\varphi_u)-\h(\varphi_\tu))\\
	&\quad -\divergence((\varphi_u-\varphi_\tu)\v_\tu'[h])-\divergence(\varphi_\tu'[h](\v_u-\v_\tu)),\\
	F_3&\coloneqq \varphi_\tu'[h]\psi'''(\xi)(\varphi_u-\varphi_\tu),\\
	F_4&\coloneqq -\sigma_\tu'[h](\h(\varphi_u)-\h(\varphi_\tu))-\varphi_\tu'[h]\big((\sigma_u-\sigma_\tu)\h'(\varphi_u) +\sigma_\tu\h'(\xi)(\varphi_u-\varphi_\tu)\big).
	\end{align*}
	Using the Lipschitz-continuity of $h(\cdot)$ together with \eqref{THM:EXS:EST}, \eqref{LEM:LIP:EST} and \eqref{EST:LIN}, a straightforward calculation shows that
	\begin{align}
	\nonumber\norm{F_i}_{L^2(L^2)}&\leq C\norm{h}_{L^2(L^2)}\norm{u-\tu}_{L^2(L^2)},\quad 1\leq i\leq 4,\\
	\label{lip_est_fre_1}\norm{\F}_{L^2(\L^2)}&\leq C\norm{h}_{L^2(L^2)}\norm{u-\tu}_{L^2(L^2)}.
	\end{align}
	With similar arguments, it follows that
	\begin{equation}
	\label{lip_est_fre_2}\norm{F_1}_{L^8(L^2)}\leq C\norm{h}_{L^2(L^2)}\norm{u-\tu}_{L^2(L^2)}.
	\end{equation}
	Now, using Gagliardo-Nirenberg's inequality, we have the continuous embedding
	\begin{equation*}
	L^{\infty}(H^1)\cap L^2(H^3)\hookrightarrow L^8(L^{\infty}).
	\end{equation*}
	Therefore, be \eqref{LEM:LIP:EST} and \eqref{EST:LIN}, it follows that
	\begin{equation}
	\label{lip_est_fre_3}\norm{\F}_{L^8(\L^2)}\leq C\norm{h}_{L^2(L^2)}\norm{u-\tu}_{L^2(L^2)}.
	\end{equation}
	Using the assumptions on $\h(\cdot)$, \eqref{THM:EXS:EST}, \eqref{LEM:LIP:EST} and \eqref{EST:LIN}, we obtain
	\begin{equation}
	\label{lip_est_fre_4}\norm{\grad F_1}_{L^8(\L^2)}\leq C\norm{h}_{L^2(L^2)}\norm{u-\tu}_{L^2(L^2)}.
	\end{equation}
	From the Lipschitz-continuity of $\h(\cdot)$ and the boundedness of $\h'(\cdot)$, applying \eqref{THM:EXS:EST}, \eqref{LEM:LIP:EST} and \eqref{EST:LIN} yields
	\begin{align}
	\nonumber \norm{F_4}_{L^{\infty}(L^2)} &\leq C\norm{\sigma_\tu'[h]}_{L^{\infty}(L^2)}\norm{\varphi_u-\varphi_\tu}_{L^{\infty}(\OmegaT)} + C\norm{\varphi_\tu'[h]}_{L^{\infty}(L^2)}\left(\norm{\sigma_u-\sigma_\tu}_{L^{\infty}(\OmegaT)}+\norm{\varphi_u-\varphi_\tu}_{L^{\infty}(\OmegaT)}\right)\\
	\label{lip_est_fre_5}&\leq C\norm{h}_{L^2(L^2)}\norm{u-\tu}_{L^2(L^2)}.
	\end{align}
	It remains to estimate the term $F_3$. Using the boundedness of $\psi'''(\xi)\in L^{\infty}(\OmegaT)$, \eqref{LEM:LIP:EST} and \eqref{EST:LIN}, we deduce that
	\begin{equation}
	\label{lip_est_fre_6}\norm{F_3}_{L^{\infty}(L^2)}\leq C\norm{\varphi_\tu'[h]}_{L^{\infty}(\OmegaT)}\norm{\varphi_u-\varphi_\tu}_{L^{\infty}(\OmegaT)}\leq C\norm{h}_{L^2(L^2)}\norm{u-\tu}_{L^2(L^2)}.
	\end{equation} 
	Using the assumptions on $\psi(\cdot)$ and the Sobolev embeddings $H^1\subset L^p,\,\H^1\subset \L^p,\,p\in [1,6]$, thanks to \eqref{THM:EXS:EST}, \eqref{LEM:LIP:EST} and \eqref{EST:LIN} we have
	\begin{align}
	\nonumber \norm{\grad F_3}_{L^2(\L^2)} &= \norm{\grad\varphi_\tu'[h]\psi'''(\xi)(\varphi_u-\varphi_\tu) + \varphi_\tu'[h]\psi^{(4)}(\xi)\grad\xi (\varphi_u-\varphi_\tu) + \varphi_\tu'[h]\psi'''(\xi)\grad (\varphi_u-\varphi_\tu)}_{L^2(\L^2)}\\
	\nonumber &\leq C\norm{ \varphi_\tu'[h]}_{L^{\infty}(H^2)}\norm{\varphi_u-\varphi_\tu}_{L^{\infty}(H^2)}\\
	\label{lip_est_fre_7}&\leq C\norm{h}_{L^2(L^2)}\norm{u-\tu}_{L^2(L^2)}.
	\end{align}
	With similar arguments, it follows that
	\begin{equation}
	\label{lip_est_fre_8}\norm{\laplace F_3}_{L^2(L^2)}\leq C\norm{h}_{L^2(L^2)}\norm{u-\tu}_{L^2(L^2)}.
	\end{equation}
	Hence, from \eqref{lip_est_fre_1}, \eqref{lip_est_fre_6}-\eqref{lip_est_fre_8} and elliptic regularity we obtain
	\begin{equation}
	\label{lip_est_fre_9}\norm{F_3}_{L^{\infty}(L^2)\cap L^2(H^2)}\leq C\norm{h}_{L^2(L^2)}\norm{u-\tu}_{L^2(L^2)}.
	\end{equation}
	Finally, using \eqref{lip_est_fre_1}-\eqref{lip_est_fre_5} and \eqref{lip_est_fre_9}, an application of Proposition \ref{PROP:LIN} yields
	\begin{equation*}
	\norm{(\varphi,\mu,\sigma,\v,p)}_{\V_1}\leq C\norm{h}_{L^2(L^2)}\norm{u-\tu}_{L^2(L^2)},
	\end{equation*}
	hence \eqref{LIP:EST:FRE:DER} holds. This completes the proof.
\end{proof}

\begin{remark}
	Since the Fréchet derivative $\S'(u)$ maps again into the space $\V_1$ and is also continuous with respect to the operator norm on $\mathcal L(L^2(L^2);\V_1)$, we conjecture that the procedure of Proposition \ref{LEM:FRE} can be repeated arbitrarily often provided that $\psi$, $\h$, $\varphi_0$, $\sigma_B$ and $\Gamma$ are smooth. Then, it were possible to show that the control-to-state operator is actually smooth. \\[1ex]
	Assuming that the control-to-state operator were at least twice continuously Fréchet differentiable, we could use this property in Section 5.3 to derive an alternative second-order sufficient condition for local optimality. However, we decided to use a different approach which is based on Fréchet differentiability of the control-to-costate operator (see Section 4) as we preferred the resulting optimality condition.
\end{remark}

\section{The adjoint state and its properties}
In optimal control theory, it is a standard approach to use \textbf{adjoint variables} to express the optimality conditions suitably. They are given by the \textbf{adjoint system} which can be derived by formal Lagrangian technique. It consists of the following equations:
\begin{subequations}
\label{EQ:ADJ}
	\begin{empheq}[left=\textnormal{(ADJ)}\hspace{-3pt}\empheqlbrace]{align}
	\label{EQ:ADJ1}
	\divergence(\w) &= 0 &&\text{in}\;\OmegaT, \\
	\label{EQ:ADJ2}
	-\eta\laplace\w + \nu\w &=  -\grad q +\varphi_u\grad\phi  &&\text{in}\;\OmegaT,\\
	\label{EQ:ADJ3}
	\delt\phi+\v_u\cdot\grad\phi &= -(P\sigma_u-A-u)\h'(\varphi_u)\phi - \h'(\varphi_u)\sigma_u \rho - \psi''(\varphi_u)\tau + \laplace\tau\notag\\
	& + \grad(\mu_u+\chi\sigma_u)\cdot\w+(P\sigma_u-A)\h'(\varphi_u)q - \upalpha_1(\varphi_u-\varphi_d) &&\text{in}\;\OmegaT,\\
	\label{EQ:ADJ4}
	\tau &= \grad\varphi_u\cdot\w + m\laplace\phi  &&\text{in}\;\OmegaT,\\
	\label{EQ:ADJ5}
	\laplace\rho - \h(\varphi_u)\rho - b \rho&= - \chi\tau + P\h(\varphi_u)\phi +\chi\grad\varphi_u\scdot\w - P\h(\varphi_u)q 
	&&\text{in}\;\OmegaT,\\[2mm]
	\label{EQ:ADJ6}
	\deln\phi &=\deln\rho=0 &&\text{on}\;\GammaT,\\
	\label{EQ:ADJ7}
	0 &= (2\eta \D\w - q\I + \varphi_u\phi\I)\n &&\text{on}\;\GammaT,\\
	\label{EQ:ADJ8}
	\deln\tau &= \phi\v_u\cdot\n - (\mu_u+\chi\sigma_u)\w\cdot\n &&\text{on}\;\GammaT,\\[2mm]
	\label{EQ:ADJ10}
	\phi(T)&=\upalpha_0(\varphi_u(T)-\varphi_f) &&\text{in}\; \Omega.
	\end{empheq}
\end{subequations}

\subsection{Existence and uniqueness of weak solutions}

A weak solution of this system, which is referred to as an \textbf{adjoint state} or \textbf{costate}, is defined as follows:
\begin{definition}
	\label{DEF:WSADJ}
	Let $u\in\UR$ be any control and let $(\varphi_u,\mu_u,\sigma_u,\v_u,p_u)$ denote its corresponding state. Then $(\phi,\tau,\rho,\w,q)$ is called a weak solution of the adjoint system \eqref{EQ:ADJ} if:
	\begin{enumerate}
		\itemi The functions $\phi,\tau,\rho,\w$ and $q$ have the following regularity:
		\begin{gather*}
		\phi \in H^1\big((H^1)^\ast\big) \cap L^\infty(H^1)\cap L^2(H^3), \;\; \tau \in L^2(H^1),\;\;
		\rho\in L^2(H^1),  \;\;\w\in L^2\big(\H^1\big), \;\; q\in L^2(L^2).
		\end{gather*}
		\itemii The quintuplet $(\phi,\tau,\rho,\w,q)$ satisfies the equations
		\begin{alignat}{2}
		\label{WF:ADJ0} 
		\phi(T) &= \upalpha_0(\varphi_u(T)-\varphi_f) &&\quad\text{a.e. in}\; \Omega,\\
		\label{WF:ADJ1} 
		\divergence(\w) &= 0 &&\quad\text{a.e. in}\; \OmegaT 
		\end{alignat}
		and
		\begin{align}
		\label{WF:ADJ2} 
		&\intO 2\eta\D\w\scolon\grad\tw + \nu\w\scdot\tw - q\divergence(\tw) \dx = \sminus\intO \phi\grad\varphi_u\scdot\tw + \phi\varphi_u \divergence(\tw) \dx , \\
		\nonumber &\big\langle \delt\phi, \tphi \big\rangle_{H^1} = -\intO  \big[ (P\sigma_u-A-u)\h'(\varphi_u)\phi  + \h'(\varphi_u)\sigma_u\rho + \psi''(\varphi_u)\tau \notag\\[-2mm]
		\nonumber& \hspace{120pt} - (P\sigma_u-A)\h'(\varphi_u) q - \phi\divergence(\v_u) +\upalpha_1(\varphi_u-\varphi_d)\big] \tphi \dx \notag\\
		\label{WF:ADJ3} & \hspace{90pt}+ \intO \big[\phi\v_u -(\mu_u+\chi\sigma_u)\w -\grad\tau\, \big]\cdot\grad\tphi\dx,\\
		\label{WF:ADJ4} 
		&\intO \tau\ttau \dx = \intO  \grad\varphi_u\cdot\w \ttau - m\grad\phi\cdot\grad\ttau \dx, \\
		\nonumber
		&-\intO \grad\rho\cdot\grad\trho \dx - b \intO \rho\trho \dS = \intO \big[ -\chi\tau+P\h(\varphi_u)\phi + \chi\grad\varphi_u\cdot\w \\[-2mm]
		\label{WF:ADJ5} & \hspace{180pt} - P\h(\varphi_u)q + \h(\varphi_u)\rho \big]\trho\dx
		\end{align}
		for a.e. $t\in (0,T)$ and all $\tphi,\ttau,\trho \in H^1,~\tw\in \H^1$.
	\end{enumerate}
\end{definition}

\bigskip

To prove existence and uniqueness of solutions for \eqref{EQ:ADJ}, we will use the following Lemma:

\begin{lemma}
	\label{LEM:ADJ}
	Let $u\in\UR$ be any control and let $(\varphi_u,\mu_u,\sigma_u,\v_u,p_u)$ denote its corresponding state. Furthermore, let $(G_0,G_1,G_2,G_3,\G_1,\G_2)\in \V_4$ be arbitrary. Then, there exists a unique solution $(\phi,\tau,\rho,\w,q)\in \V_3$ solving
	\begin{subequations}
		\label{EQ:ADJG}
		\begin{align}
		\label{EQ:ADJ2:1}
		\divergence(\w) &= 0 &&\text{a.e. in}\;\OmegaT, \\
		\label{EQ:ADJ2:2}
		-\eta\laplace\w + \nu\w &=  -\grad q +\varphi_u\grad\phi +  \G_1 &&\text{a.e. in}\;\OmegaT,\\
		\label{EQ:ADJ2:4}
		\tau &= \grad\varphi_u\cdot\w + m\laplace\phi + G_2 &&\text{a.e. in}\;\OmegaT,\\
		\label{EQ:ADJ2:5}
		\laplace\rho - b\rho - \h(\varphi_u)\rho &= - \chi\tau + P\h(\varphi_u)\phi +\chi\grad\varphi_u\scdot\w - P\h(\varphi_u)q + G_3
		&&\text{a.e. in}\;\OmegaT,\\[2mm]
		\label{EQ:ADJ2:6}
		\deln\phi &=\deln\rho=0 &&\text{a.e. on}\;\GammaT,\\
		\label{EQ:ADJ2:7}
		0 &= (2\eta \D\w - q\I + \varphi_u\phi\I)\n &&\text{a.e. on}\;\GammaT,\\
		\label{EQ:ADJ2:8}
		\phi(T)&=G_0 &&\text{a.e. in}\; \Omega,
		\end{align}
		and 
		\begin{align}
		\label{EQ:ADJ2:3} 
		&\big\langle \delt\phi, \tphi \big\rangle_{H^1} = -\intO  \big[ (P\sigma_u-A-u)\h'(\varphi_u)\phi  + \h'(\varphi_u)\sigma_u\rho + \psi''(\varphi_u)\tau \notag\\[-2mm]
		& \hspace{120pt} - (P\sigma_u-A)\h'(\varphi_u) q - \phi\divergence(\v_u) + G_1\big] \tphi \dx \notag\\
		& \hspace{90pt}+ \intO \big[\phi\v_u -(\mu_u+\chi\sigma_u)\w -\grad\tau+\G_2\, \big]\cdot\grad\tphi\dx,
		\end{align}
		for a.e. $t\in (0,T)$ and all $\tphi \in H^1$.
	\end{subequations}	
	In addition, it holds that
	\begin{align}
	\label{LEM:ADJ:EST}\norm{(\phi,\tau,\rho,\w,q)}_{\V_3} \leq C\norm{(G_0,G_1,G_2,G_3,\G_1,\G_2)}_{\V_4}.
	\end{align}
\end{lemma}

\begin{proof}[Proof of Lemma \ref{LEM:ADJ}] 
	The proof is a straightforward modification of the proof of \cite[Thm. 16]{EbenbeckKnopf}. We will only sketch the main differences. Testing \eqref{EQ:ADJ2:3} with $\tphi = \laplace\phi-\phi$, we have to estimate the terms
	\begin{equation*}
	-\intO G_1(\laplace\phi-\phi)\dx,\qquad \intO \G_2\cdot\grad(\laplace\phi-\phi)\dx.
	\end{equation*}
	Using the continuous embedding $H^1\hookrightarrow L^6$ together with Young's and Hölder's inequalities, these two terms can be controlled via
	\begin{equation*}
	\left|\intO -G_1(\laplace\phi-\phi) + \G_2\cdot\grad(\laplace\phi-\phi)\dx \right| \leq C\left(\normh{1}{\phi}^2 + \norml{\frac{6}{5}}{G_1}^2+\normL{2}{\G_2}^2\right) + \frac{m}{16}\left(\norml{2}{\laplace\phi}^2 + \normL{2}{\grad\laplace\phi}^2\right).
	\end{equation*}
	The last two terms on the right-hand side of this inequality can be absorbed into the left-hand side of an energy inequality whereas the first term on the right-hand side can be controlled through a Gronwall term. The terms involving $G_1$ and $\G_2$ enter the inequality \eqref{LEM:ADJ:EST}. Apart from these arguments, the remaining estimates can be carried out with straightforward modifications of the proof of \cite[Thm. 16]{EbenbeckKnopf}
\end{proof}

\smallskip

\begin{corollary}\label{COR:ADJ}
	Let $u\in\UR$ be any control and let $(\varphi_u,\mu_u,\sigma_u,\v_u,p_u)$ denote its corresponding state. Then, there exists a unique weak solution $(\phi_u,\tau_u,\rho_u,\w_u,q_u)\in \V_3$ of \eqref{EQ:ADJ} in the sense of Definition \ref{DEF:WSADJ}. 
\end{corollary}
\begin{proof}
	This follows from an application of Lemma \ref{LEM:ADJ} with the following choices:
	\begin{gather*}
	\G_1= \mathbf{0},\qquad \G_2 = \mathbf{0},\qquad G_0=\upalpha_0(\varphi_u(T)-\varphi_f),\qquad
	G_1=\upalpha_1(\varphi_u\sminus \varphi_d),\qquad 
	G_2=0,
	\qquad G_3=0.
	\end{gather*}
	Since $\varphi_u\in C(H^2)$, it follows that $\varphi_u(T)\in H^1$ with bounded norm. Hence, it is easy to check that $(G_0,G_1,G_2,G_3,\G_1,\G_2)\in \V_4$ with bounded norm. Moreover, using \eqref{EQ:ADJ2:1}-\eqref{EQ:ADJ2:8}, it is straightforward to check that \eqref{WF:ADJ0}-\eqref{WF:ADJ2} and \eqref{WF:ADJ4}-\eqref{WF:ADJ5} are fulfilled. This completes the proof.
\end{proof}

\bigskip

Similar to the definition of the control-to-state operator, we can define an operator that maps any control $\u\in\UR$ onto its corresponding adjoint state:

\begin{definition}
	We define the \textbf{control-to-costate operator} $\A\colon \UR\to \V_3$ as the operator assigning to every $\u\in\UR$ the unique weak solution $(\phi_u,\tau_u,\rho_u,\w_u,q_u)\in \V_3$ of the adjoint system \eqref{EQ:ADJ}.
\end{definition}

\subsection{Lipschitz continuity}

In the following, we show that the control-to-costate operator is Lipschitz-continuous:
\begin{proposition}\label{PROP:ADJ:LIP}
	There exists some constant $L_3>0$ depending only on $R$, $\Omega$ and $T$ such that for all $u,\tu\in\UR$,
	\begin{equation}
	\label{ADJ:LIP:EST}
	\norm{\A(\tu)-\A(u)}_{\V_3}\leq L_3\norm{\tu-u}_{L^2(L^2)}.
	\end{equation}
\end{proposition}
\begin{proof}
	We first define
	\begin{equation*}
	(\phi_,\tau,\rho,\w,q)\coloneqq \S(\tu) - \S(u) = (\phi_\tu,\tau_\tu,\rho_\tu,\w_\tu,q_\tu)-(\phi_u,\tau_u,\rho_u,\w_u,q_u)
	\end{equation*}
	and introduce the variable
	\begin{equation*}
	\pi \coloneqq q-\phi_\tu (\varphi_\tu-\varphi_u).
	\end{equation*}
	Then, the quintuplet $(\phi,\tau,\rho,\w,\pi)$ fulfills \eqref{EQ:ADJG} with
	\begin{align*}
	\G_1&= -\phi_\tu\grad (\varphi_\tu-\varphi_u),\\
	\G_2&= \phi_\tu(\v_\tu-\v_u) - \w_\tu\big((\mu_\tu+\chi\sigma_\tu)-(\mu_u+\chi\sigma_u)\big),\\
	G_0 &= \upalpha_0(\varphi_\tu(T)-\varphi_u(T)),\displaybreak\\
	G_1 &= \phi_\tu [\big( P(\sigma_\tu-\sigma_u)-(\tu-u)\big)\h'(\varphi_u)] + \phi_\tu (P\sigma_\tu-A-\tu)\big(\h'(\varphi_\tu)-\h'(\varphi_u)\big)\\
	&\quad + \rho_\tu \big(\h'(\varphi_u)(\sigma_\tu-\sigma_u) + \sigma_\tu(\h'(\varphi_\tu)-\h'(\varphi_u))\big)+(\psi''(\varphi_\tu)-\psi''(\varphi_u))\tau_\tu\\
	&\quad - q_\tu[P(\sigma_\tu-\sigma_u)\h'(\varphi_u) + (P\sigma_\tu-A)(\h'(\varphi_\tu)-\h'(\varphi_u))] - (P\sigma_u-A)\h'(\varphi_u)\phi_\tu(\varphi_\tu-\varphi_u) \\
	&\quad -\phi_\tu\divergence(\v_\tu-\v_u) 
	+ \upalpha_1(\varphi_\tu-\varphi_u), \\
	G_2 &= \grad(\varphi_\tu-\varphi_u)\cdot\w_\tu,\\
	G_3 &= \rho_\tu(\h(\varphi_\tu)-\h(\varphi_u)) + P\phi_\tu(\h(\varphi_\tu)-\h(\varphi_u)) + \chi\grad(\varphi_\tu-\varphi_u)\cdot\w_\tu,\\
	&\quad  - P\h(\varphi_u)\phi_\tu(\varphi_\tu-\varphi_u) - Pq_\tu(\h(\varphi_\tu)-\h(\varphi_u)). 
	\end{align*}
	Using \eqref{THM:EXS:EST}, \eqref{LEM:LIP:EST} and the mean value theorem, it is easy to check that
	\begin{equation*}
	\norm{\psi''(\varphi_\tu)-\psi''(\varphi_u)}_{L^{\infty}(L^{\infty})}\leq C\norm{\varphi_\tu-\varphi_u}_{L^{\infty}(L^{\infty})},\quad \norm{\h'(\varphi_\tu)-\h'(\varphi_u)}_{L^{\infty}(L^{\infty})}\leq C\norm{\varphi_\tu-\varphi_u}_{L^{\infty}(L^{\infty})}.
	\end{equation*}
	Then, using Proposition \ref{THM:EXS}, Corollary \ref{COR:LIP} and Corollary \ref{COR:ADJ}, a straightforward calculation shows that 
	\begin{equation*}
	\norm{(G_0,G_1,G_2,G_3,\G_1,\G_2)}_{\V_4}\leq C\norm{\tu-u}_{L^2(L^2)}.
	\end{equation*}
	Consequently, the estimate \eqref{LEM:ADJ:EST} implies that
	\begin{equation}
	\norm{(\phi,\tau,\rho,\w,\pi)}_{\V_3}\leq C\norm{(G_0,G_1,G_2,G_3,\G_1,\G_2)}_{\V_4}\leq C\norm{\tu-u}_{L^2(L^2)}.
	\end{equation}
	Recalling the definitions of $\pi$ and $\V_3$, it remains to show that
	\begin{equation*}
	\norm{\phi_\tu(\varphi_\tu-\varphi_u)}_{L^2(H^1)}\leq C\norm{\tu-u}_{L^2(L^2)}.
	\end{equation*}
	However, this is another easy consequence of Corollary \ref{COR:LIP} and Corollary \ref{COR:ADJ}. Therefore, it follows that
	\begin{equation}
	\norm{(\phi,\tau,\rho,\w,q)}_{\V_3}\leq C\norm{\tu-u}_{L^2(L^2)},
	\end{equation}
	which completes the proof.
\end{proof}

\subsection{Fréchet differentiability}

We can also show that the control-to-costate operator is continuously Fréchet differentiable:

\pagebreak[2]

\begin{proposition}
	\label{LEM:FREADJ} 
	The following statements hold:
	\begin{enumerate}
		\itemi The control-to-costate operator $\A$ is Frechét-differentiable on $\UR$, i.e., for any $u\in\UR$ there exists a unique bounded, linear operator 
		\begin{equation*}
		\A'(u)\colon L^2(L^2)\to\V_3, \quad h\mapsto \A'(u)[h]= (\phi_u',\tau_u',\rho_u',\w_u',q_u')[h],
		\end{equation*}
		such that
		\begin{equation*}
		\frac{\norm{\A(u+h)-\A(u)-\A'(u)[h]}_{\V_3}}{\norm{h}_{L^2(L^2)}}\to 0\quad \text{as }\norm{h}_{L^2(L^2)}\to 0.
		\end{equation*}
		For any $\u\in \U$ and $h\in L^2(L^2)$, the Fréchet-derivative $(\phi_u',\tau_u',\rho_u',\w_u',q_u')[h]$ is the unique solution of \eqref{EQ:ADJG} with 
		\begin{align*}
		G_0 &= \upalpha_0\varphi_u'[h](T),\\
		G_1&= (P\sigma_u-A-u)\h''(\varphi_u)\varphi_u'[h]\phi_u + \h'(\varphi_u)(P\sigma_u'[h]-h)\phi_u\\
		&\quad + \big(\h'(\varphi_u)\sigma_u'[h]+\h''(\varphi_u)\varphi_u'[h]\sigma_u\big)\rho_u + \psi^{(3)}(\varphi_u)\varphi_u'[h]\tau_u\\
		&\quad - \big(P\sigma_u'[h]\h'(\varphi_u) + (P\sigma_u-A)\h''(\varphi_u)\varphi_u'[h]\big)q_u - \phi_u\divergence(\v_u'[h]) +\upalpha_1\varphi_u'[h],\\
		G_2&= \grad\varphi_u'[h]\cdot\w_u ,\\
		G_3&= \h'(\varphi_u)\varphi_u'[h]\rho_u + P\h'(\varphi_u)\varphi_u'[h]\phi_u + \chi\grad\varphi_u'[h]\cdot \w_u - P\h'(\varphi_u)\varphi_u'[h]q_u,\\
		\G_1&= \varphi_u'[h]\grad\phi_u ,\\
		\G_2&= \phi_u\v_u'[h] - (\mu_u'[h]+\chi\sigma_u'[h]) \w_u
		\end{align*}
		and \eqref{EQ:ADJ2:7} replaced by
		\begin{equation}
		\label{EQ:ADJ3:7}(2\eta\D\w-q\I+\varphi_u\phi\I + \varphi_u'[h]\phi_u\I)\n = 0\quad \text{a.e. on }\GammaT.
		\end{equation}
		\itemii The Frechet-derivative is Lipschitz continuous, i.e., for any $u,\tu\in \UR$ and $h\in L^2(L^2)$, it holds that
		\begin{equation}
		\label{LIP:EST:FRE:DER:2} \norm{\A'(u)[\cdot]-\A'(\tu)[\cdot]}_{\mathcal L(L^2(L^2);\V_3)} \leq L_4 \norm{u-\tu}_{L^2(L^2)},
		\end{equation}
		with a constant $L_4>0$ depending only on the system parameters and on $\Omega$, $T$, $R$.
	\end{enumerate}
\end{proposition}

\begin{proof} 
	The proof proceeds similarly to the proof of Proposition \ref{LEM:FRE}.
	\paragraph{Proof of \textnormal{(i)}:}
	\textit{Step 1:} Existence of a solution to \eqref{EQ:ADJG} with the above choices for $(G_0,G_1,G_2,G_3,\G_1,\G_2)$ follows from a simple pressure reformulation argument. Indeed, let us define
	\begin{align*}
	&\tilde{\G}_1= -\phi_u\grad\varphi_u'[h],
	\quad\tilde{G}_1 = G_1 - (P\sigma_u-A)\h'(\varphi_u)\varphi_u'[h]\phi_u,\\
	&\tilde{G}_3 = G_3 - P\h(\varphi_u)\varphi_u'[h]\phi_u,
	\quad\tilde{G}_0= G_0,\quad \tilde{G}_2 = G_2,\quad \tilde{\G}_2 = \G_2.
	\end{align*}
	Using Proposition \ref{THM:EXS}, Proposition \ref{LEM:FRE} and Lemma \ref{LEM:ADJ}, it is straightforward to check that $(\tilde{G}_0,\tilde{G}_1,\tilde{G}_2,\tilde{G}_3,\tilde{\G}_1,\tilde{\G}_3)$ $\in \V_4$ with bounded norm. Therefore, there exists a unique weak solution $(\phi,\tau,\rho,\w,\pi)\in \V_3$ of \eqref{EQ:ADJG} according to $(G_0,G_1,G_2,G_3,\G_1,\G_2)=(\tilde{G}_0,\tilde{G}_1,\tilde{G}_2,\tilde{G}_3,\tilde{\G}_1,\tilde{\G}_3)\in \V_4$. We now define
	\begin{equation*}
	q = \pi + \varphi_u'[h]\phi_u. 
	\end{equation*}
	Using Proposition \ref{THM:EXS} and Lemma \ref{LEM:ADJ}, it follows that $\varphi_u'[h]\phi_u\in L^2(H^1)$ with bounded norm. Therefore, it follows that $(\phi,\tau,\rho,\w,q)\in \V_3$ is a weak solution of \eqref{EQ:ADJG} with $(G_0,G_1,G_2,G_3,\G_1,\G_2)$ as above and \eqref{EQ:ADJ2:7} replaced by \eqref{EQ:ADJ3:7}. Uniqueness of solutions of this system follows due to linearity of the system and estimate \eqref{LEM:ADJ:EST}. \\[1ex]
	\textit{Step 2:} In the following, we define
	\begin{equation*}
	(\phi_\CR^h,\tau_\CR^h,\rho_\CR^h,\w_\CR^h,q_\CR^h) = (\phi_{u+h},\tau_{u+h},\rho_{u+h},\w_{u+h},q_{u+h})-(\phi_u,\tau_u,\rho_u,\w_u,q_u)-(\phi_u'[h],\tau_u'[h],\rho_u'[h],\w_u'[h],q_u'[h]).
	\end{equation*}
	Then, we can check that $(\phi_\CR^h,\tau_\CR^h,\rho_\CR^h,\w_\CR^h,q_\CR^h)$ is the solution of \eqref{EQ:ADJG} with 
	\begin{align*}
	G_0 &= \upalpha_0\varphi_\CR^h(T),\\
	G_1&= (P\sigma_u-A-u)\left[\h''(\varphi_u)\varphi_\CR^h + \frac{1}{2}\h^{(3)}(\xi)(\varphi_{u+h}-\varphi_u)^2\right]\phi_u + P\sigma_{\CR}^h\h'(\varphi_u)\phi_u\\
	&\quad +  [P(\sigma_{u+h}-\sigma_u)-h](\h'(\varphi_{u+h})-\h'(\varphi_u))\phi_u\\
	&\quad + (\phi_{u+h}-\phi_u)\left[(P\sigma_u-A-u)(\h'(\varphi_{u+h})-\h'(\varphi_u)) + \big(P(\sigma_{u+h}-\sigma_u)-h\big)\h'(\varphi_{u+h})\right]\\
	&\quad + \rho_u\left[\h''(\varphi_u)\varphi_\CR^h\sigma_u +  \frac{1}{2}\h^{(3)}(\xi)(\varphi_{u+h}-\varphi_u)^2\,\sigma_u+ \h'(\varphi_u)\sigma_{\CR}^h +(\h'(\varphi_{u+h})-\h'(\varphi_u))(\sigma_{u+h}-\sigma_u)\right]\\
	&\quad +(\rho_{u+h}-\rho_u)\left[(\h'(\varphi_{u+h})-\h'(\varphi_u))\sigma_u + \h'(\varphi_{u+h})(\sigma_{u+h}-\sigma_u)\right]\\
	&\quad + \tau_u\left[\psi^{(3)}(\varphi_u)\varphi_{\CR}^h + \psi^{(4)}(\xi)(\varphi_{u+h}-\varphi_u)^2\right] + (\tau_{u+h}-\tau_u)(\psi''(\varphi_{u+h})-\psi''(\varphi_u))\\
	&\quad -(P\sigma_u-A)\left[\h''(\varphi_u)\varphi_\CR^h + \frac{1}{2}\h^{(3)}(\xi)(\varphi_{u+h}-\varphi_u)^2\right]q_u - P\sigma_{\CR}^h\h'(\varphi_u)q_u\\
	&\quad - P(\sigma_{u+h}-\sigma_u)(\h'(\varphi_{u+h})-\h'(\varphi_u))q_u\\
	&\quad - (q_{u+h}-q_u)\left[(P\sigma_u-A)(\h'(\varphi_{u+h})-\h'(\varphi_u)) + P(\sigma_{u+h}-\sigma_u)\h'(\varphi_{u+h})\right]\\
	&\quad - \phi_u\divergence(\v_{\CR}^h) - (\phi_{u+h}-\phi_u)\divergence(\v_{u+h}-\v_u) + \upalpha_1\varphi_\CR^h,\\
	G_2&= \grad\varphi_\CR^h\cdot \w_u + \grad(\varphi_{u+h}-\varphi_u)\cdot(\w_{u+h}-\w_u),\\
	G_3&= [\h'(\varphi_u)\varphi_\CR^h + \frac{1}{2}\h''(\xi)(\varphi_{u+h}-\varphi_u)^2]\rho_u + (\h(\varphi_{u+h})-\h(\varphi_u))(\rho_{u+h}-\rho_u)\\
	&\quad + P[\h'(\varphi_u)\varphi_{\CR}^h + \frac{1}{2}\h''(\xi)(\varphi_{u+h}-\varphi_u)^2](\phi_u-q_u)\\
	&\quad + P(\h(\varphi_{u+h})-\h(\varphi_u))[(\phi_{u+h}-\phi_u)-(q_{u+h}-q_u)]\\
	&\quad + \chi\grad\varphi_\CR^h\cdot (\w_{u+h}-\w_u) + \chi\grad(\varphi_{u+h}-\varphi_u)\cdot(\w_{u+h}-\w_u) ,\\ 
	\G_1&= \varphi_\CR^h\grad\phi_u + (\varphi_{u+h}-\varphi_u)\grad(\phi_{u+h}-\phi_u),\\
	\G_2&= \phi_u\v_{\CR}^h + (\phi_{u+h}-\phi_u)(\v_{u+h}-\v_u) - (\mu_{\CR}^h+\chi\sigma_{\CR}^h)\w_u \\
	&\quad - [(\mu_{u+h}+\chi\sigma_{u+h})-(\mu_u+\chi\sigma_u)](\w_{u+h}-\w_u)
	\end{align*}
	and \eqref{EQ:ADJ2:7} replaced by
	\begin{equation}
	\label{EQ:ADJ34:7} \left(2\eta\D\w - q\I + \varphi_u\phi\I + \varphi_\CR^h\phi_u\I + (\varphi_{u+h}-\varphi_u)(\phi_{u+h}-\phi_u)\I\right)\n = 0\quad \text{a.e. on }\GammaT.
	\end{equation}
	We now introduce a new pressure
	\begin{equation*}
	\pi_\CR^h = q_\CR^h + \varphi_\CR^h\phi_u + (\varphi_{u+h}-\varphi_u)(\phi_{u+h}-\phi_u)
	\end{equation*}
	and we define
	\begin{align*}
	\tilde{\G}_1&= -\phi_u\grad\varphi_\CR^h - (\phi_{u+h}-\phi_u)\grad(\varphi_{u+h}-\varphi_u) ,\\
	\tilde{G}_1 &= G_1 - (P\sigma_u-A)\h'(\varphi_u)\left[\varphi_\CR^h\phi_u + (\varphi_{u+h}-\varphi_u)(\phi_{u+h}-\phi_u)\right],\\
	\tilde{G}_3 &= G_3 - P\h(\varphi_u)\left[\varphi_\CR^h\phi_u + (\varphi_{u+h}-\varphi_u)(\phi_{u+h}-\phi_u)\right],\\
	\tilde{G}_0&= G_0,\quad \tilde{G}_2 = G_2,\quad \tilde{\G}_2 = \G_2.
	\end{align*}
	Then, we can check that $(\phi_\CR^h,\tau_\CR^h,\rho_\CR^h,\w_\CR^h,\pi_\CR^h)$ is a solution of \eqref{EQ:ADJG} according to $(G_0,G_1,G_2,G_3,\G_1,\G_2)=(\tilde{G}_0,\tilde{G}_1,\tilde{G}_2,\tilde{G}_3,\tilde{\G}_1,\tilde{\G}_2)$.\\[1ex]
	\textit{Step 3:} Using Proposition \ref{THM:EXS}, Corollary \ref{COR:LIP}, Proposition \ref{LEM:FRE}, Lemma \ref{LEM:ADJ} and Corollary \ref{COR:ADJ}, it can be checked that 
	\begin{align}
	\label{fre_adj_1}\norm{(\phi_\CR^h,\tau_\CR^h,\rho_\CR^h,\w_\CR^h,\pi_\CR^h)}_{\V_3} \leq C\norm{(\tilde{G}_0,\tilde{G}_1,\tilde{G}_2,\tilde{G}_3,\tilde{\G}_1,\tilde{\G}_2)}_{\V_4}\leq C\norm{h}_{L^2(L^2)}^2.
	\end{align}
	Using this inequality together with Corollary \ref{COR:LIP}, Proposition \ref{LEM:FRE}, Lemma \ref{LEM:ADJ} and Corollary \ref{COR:ADJ}, recalling the definition of $\V_3$ and the expression for $\pi_\CR^h$, it follows that
	\begin{align*}
	\norm{q_\CR^h}_{L^2(H^1)} &\leq \norm{\pi_\CR^h}_{L^2(H^1)} + \norm{\phi_u}_{L^{\infty}(H^1)}\norm{\varphi_\CR^h}_{L^{\infty}(H^2)}+ \norm{\phi_{u+h}-\phi_u}_{L^{\infty}(H^1)}\norm{\varphi_{u+h}-\varphi_u}_{L^{\infty}(H^2)}\\
	&\leq C\norm{h}_{L^2(L^2)}^2.
	\end{align*}
	In summary, we obtain 
	\begin{align}
	\label{fre_adj_2}\norm{(\phi_\CR^h,\tau_\CR^h,\rho_\CR^h,\w_\CR^h,q_\CR^h)}_{\V_3} \leq  C\norm{h}_{L^2(L^2)}^2.
	\end{align}
	\paragraph{Proof of \textnormal{(ii)}:} Since the operator $\S'(\cdot)[h]\colon \UR\to \V_1$ is Lipschitz-continuous for all $h\in L^2(L^2)$, the proof follows with similar arguments as the proof of Proposition \ref{PROP:ADJ:LIP}. 
\end{proof}

\section{The optimal control problem}

In this section we analyze the optimal control problem that was motivated in the introduction: We intend to minimize the \textit{\bfseries cost functional}
\begin{align*}
I(\varphi,u)&:= 
\frac{\upalpha_0} 2 \norml{2}{\varphi(T)-\varphi_f}^2
+ \frac{\upalpha_1} 2 \norm{\varphi-\varphi_d}_{L^2(L^2)}^2+ \frac{\kappa} 2\norm{u}_{L^2(L^2)}^2
\end{align*}
subject to the following conditions:
\begin{itemize}
	\item $u$ is an admissible control, i.e., $u\in\U$,
	\item $(\varphi,\mu,\sigma,\v,p)$ is a strong solution of the system \eqref{EQ:CHB} to the control $u$.
\end{itemize}
Using the control-to-state operator we can formulate this optimal control problem alternatively as
\begin{align}
	\label{OCP}
		\text{Minimize} &\quad J(u) \quad \text{s.t.}\quad u\in\U,
\end{align}
where the \textit{\bfseries reduced cost functional} $J$ is defined by
\begin{align}
\label{DEF:J}
	J(u):=I\big([\S_1(u)]_1,u\big)=I(\varphi_u,u), \quad u\in\U.
\end{align}
A globally/locally optimal control of this optimal control problem is defined as follows:

\bigskip

\begin{definition}
	Let $u\in\U$ be any admissible control.
	\begin{itemize}
		\itema $\u$ is called a \textbf{(globally) optimal control} of the problem \eqref{OCP} if $J(\u)\le J(u)$ for all $u\in\U$.
		\itemb $\u$ is called a \textbf{locally optimal control} of the problem \eqref{OCP} if there exists some $\delta>0$ such that  $J(\u)\le J(u)$ for all $u\in\U$ with $\norm{u-\u}_{L^2(L^2)}<\delta$.
	\end{itemize}
	In this case, $\S(\u)$ is called the corresponding \textbf{globally/locally optimal state}.
\end{definition}

\smallskip

\subsection{Existence of a globally optimal control}

Of course, the optimal control problem \eqref{OCP} does only make sense if there exists at least one globally optimal solution. This is established by the following Theorem:

\begin{theorem}
	\label{THM:EXC}
	The optimization problem \eqref{OCP} possesses a globally optimal solution.
\end{theorem}
\begin{proof} 
	The assertion can be proved by the direct method in the calculus of variations using the basics established in Section 3. A very similar proof can be found in \cite[Thm. 17]{EbenbeckKnopf}.
\end{proof}

\subsection{First-order necessary conditions for local optimality}

Obviously, Theorem \ref{THM:EXC} does not provide uniqueness of the globally optimal control $\u$. As the control-to-state operator is nonlinear we cannot expect the cost functional to be convex. Therefore, it is possible that the optimization problem has several locally optimal controls or even several globally optimal controls. In the following, since numerical methods will (in general) only detect local minimizers, our goal is to characterize locally optimal controls by necessary optimality conditions.\\[1ex]
Since the control-to-state operator is Fréchet differentiable according to Proposition \ref{LEM:FRE}, Fréchet differentiability of the cost functional easily follows by chain rule. If $\u\in\U$ is a locally optimal control, it must hold that $J'(\u)[u-\u]\ge 0$ for all $u\in \U$. The Fréchet derivative $J'(\u)$ can be described by means of the so-called adjoint state that was introduced in Section 4.\\[1ex]
In the following we characterize locally optimal controls of \eqref{OCP} by necessary conditions which are par\-ticularly important for computational methods. The adjoint variables can be used to express the variational inequality in a very concise form:

\begin{theorem}
	Let $\u\in\U$ be a locally optimal control of the minimization problem \eqref{OCP}. Then $\u$ satisfies the \textbf{variational inequality} that is
	\begin{align}
	\label{VIQ}
		J'(\u)[u-\u] = \intOT \big[\kappa\u -\phi_\u\, \h(\varphi_\u) \big] (u-\u) \dxt \ge 0\quad\text{for all}\; u\in\U.
	\end{align}
\end{theorem}

\begin{proof} 
	The assertion is a standard result from optimal control theory. It can be proved by slight modifications of the proof of \cite[Thm. 17]{EbenbeckKnopf}. 
\end{proof}

As our set of admissible controls is a box-restricted subset of $L^2(L^2)$, a locally optimal control $\u$ can also be characterized by a projection of $\frac 1 \kappa\, \phi_\u\, \h(\varphi_\u)$ onto the set $\U$. 

\begin{corollary}
	\label{COR:PRF}
	Let $\u\in\U$ be a locally optimal control of the minimization problem \eqref{OCP}. Then $u$ is given implicitly by the \textbf{projection formula}
	\begin{align}
		\label{PRF}
		\u(x,t) = \mathbb P_{[\bar{a}(x,t),\bar{b}(x,t)]}\left(\frac 1 \kappa \,\phi_\u(x,t)\, \h\big(\varphi_\u(x,t)\big)\right)
		\quad \text{for almost all $(x,t)\in\Omega_T$},
	\end{align}
	where the projection $\mathbb P$ is defined by
	\begin{align*}
		\mathbb P_{[\zeta,\xi]}(s) = \min\big\{\xi,\max\{\zeta,s\}\big\},\qquad s\in\R,
	\end{align*}
		for any $\zeta,\xi\in\R$ with $\zeta\le\xi$. This constitutes another necessary condition for local optimality.
\end{corollary}

Since this is a well-known inference of the necessary optimality condition provided by the variational inequality, we omit the proof. For a similar proof we refer to \cite[pp.\,71-73]{troeltzsch}.

\begin{remark}
	The necessary optimality conditions \eqref{VIQ} and \eqref{PRF} are equivalent (cf. \cite[pp.\,71-73]{troeltzsch}).
\end{remark}

\subsection{A second-order sufficient condition for strict local optimality}

We also want to establish a sufficient condition for (strict) local optimality. Since the control-to-state operator $\S:\U_R\to \V_1$ and the control-to-costate operator $\A:\U_R\to \V_2$ are continuously Fréchet differentiable, so is the cost functional $J$ due to chain rule. \\[1ex]
Therefore we can easily establish a sufficient condition for strict local optimality: 
Let $\u\in\U$ satisfy the variational inequality \eqref{VIQ} (or the projection formula \eqref{PRF}, respectively) and we assume that $J''(\u)$ is positive definite, i.e., 
\begin{align}
\label{COND:POSDEF}
	J''(\u)[h,h] > 0 
\end{align}
for all directions $h\in L^2(L^2)\setminus\{0\}$. Then $\u$ is a strict local minimizer of $J$ on the set $\U$.\\[1ex]
However, this condition is far too restrictive as it suffices to require \eqref{COND:POSDEF} only for a certain class of critical directions. Such a condition for optimal control problems with general semilinear elliptic or parabolic PDE constraints was firstly established in \cite{TroeltzschCasasReyes}. Meanwhile, it can also be found, for instance, in the textbook \cite[pp.\,245-248]{troeltzsch}. We proceed similarly and define the cone of critical directions as follows:

\begin{definition}
	For $\u\in\U$, we define the set
	\begin{align}
	\label{DEF:KRITCONE}
	\C(\u) := \big\{ h \in L^2(L^2) \;\big\vert\; h\; \text{satisfies condition \eqref{COND:KRIT}} \big\}
	\end{align}
	where condition \eqref{COND:KRIT} reads as follows:
	\begin{align}
	\label{COND:KRIT}
	\text{For allmost all $(x,t)\in\OmegaT$}:\quad 
	h(x,t) \begin{cases}
	\ge 0, & \u(x,t) = \bar{a}(x,t), \\
	\le 0, & \u(x,t) = \bar{b}(x,t), \\
	= 0, & \kappa\u(x,t) - \phi_\u(x,t)\, \h(\varphi_\u)(x,t) \neq 0.
	\end{cases}
	\end{align}
	This set $\C(\u)$ is called the \textbf{cone of critical directions}.
\end{definition}

\bigskip

Now, we can use the cone $\C(\u)$ to formulate a sufficient condition for strict local optimality:

\begin{theorem}
	\label{THM:SUFF}
	Let $\u\in\U$ be any control satisfying the variational inequality \eqref{VIQ}. Moreover, we assume that $J''(\u)[h,h] > 0$ which is equivalent to
	\begin{align}
	\label{COND:SUFF}
	\intOT \big(\phi'_\u[h] \,\h(\varphi_\u) + \phi_\u \,\h'(\varphi_\u) \varphi'_\u[h] \big)\, h \dxt < \kappa \norm{h}_{L^2(L^2)}^2 \qquad\text{for all $h\in\C(\u)\setminus\{0\}$.}
	\end{align}
	Then $\u$ satisfies a quadratic growth condition, i.e., there exist $\delta,\theta>0$ such that for all $u\in\U$ with $\norm{u-\u}_{L^2(L^2)}<\delta$,
	\begin{align}
	\label{COND:QUAD}
	J(u) \ge J(\u) + \frac \theta 2 \norm{u-\u}_{L^2(L^2)}^2.
	\end{align}
	In particular, this means that $\u$ is a strict local minimizer of the functional $J$ on the set $\U$.
\end{theorem}

\begin{proof}
	The second-order Fréchet derivative of $J$ is given by
	\begin{align}
	\label{EQ:FD2}
		J''(\u)[h_1,h_2] = 
		\kappa\, \big(h_1\,{,}\,h_2 \big)_{L^2(L^2)} - \intOT \big(\phi'_\u[h_1] \,\h(\varphi_\u) + \phi_\u \,\h'(\varphi_\u) \varphi'_\u[h_1] \big)\, h_2 \dxt
	\end{align}
	for all $h_1,h_2\in L^2(L^2)$. Thus, condition \eqref{COND:SUFF} is equivalent to 
	\begin{align}
	\label{EQ:SUFF:1}
		J''(\u)[h,h] \ge 0 \quad \qquad\text{for all $h\in\C(\u)\setminus\{0\}$.}
	\end{align}
	Studying the proof of \cite[Lem.\,4.1]{TroeltzschCasasReyes} carefully, we find out that our theorem can be proved analogously if the following assertions can be established:
	\begin{enumerate}
		\itemi For any sequences $(u_k)\subset\U$ and $(h_k)\subset L^2(L^2)$ with $u_k\to\u$ and $h_k\wto h$ in $L^2(L^2)$, it holds that
		\begin{align}
		\label{EQ:SUFF:A1}
			J'(u_k)[h_k] \to J'(\u)[h] \qquad \text{as}\; k\to\infty.
		\end{align}
		\itemii For any sequence  $(h_k)\subset L^2(L^2)$ with $h_k\wto h$ in $L^2(L^2)$, it holds that
		\begin{align}
		\label{EQ:SUFF:A2}
			\B(h_k,h_k) \to \B(h,h) \qquad \text{as}\; k\to\infty
		\end{align}
		up to subsequence extraction where the bilinear form $\B$ is defined by
		\begin{align*}
			\B: L^2(L^2)\times L^2(L^2) \to \R, \quad (h_1,h_2) 
			\mapsto \intOT \big(\phi'_\u[h_1] \,\h(\varphi_\u) + \phi_\u \,\h'(\varphi_\u) \varphi'_\u[h_1] \big)\, h_2 \dxt.
		\end{align*}
	\end{enumerate}
	\paragraph{Proof of \textnormal{(i)}:} Let $(u_k)\subset\U$ and $(h_k)\subset L^2(L^2)$ with $u_k\to\u$ and $h_k\wto h$ in $L^2(L^2)$ be arbitrary. Recall that the first-order Fréchet derivative of $J$ is given by
	\begin{align}
	\label{EQ:FD1}
		J'(u)[h] = \intOT \big[\kappa u -\phi_u\, \h(\varphi_u) \big]\, h \dxt ,\qquad u\in \U,\,h\in L^2(L^2).
	\end{align}
	Since $\kappa \u -\phi_\u\, \h(\varphi_\u)$ lies in $L^2(L^2)$ it directly follows that $J'(\u)[h_k] \to J'(\u)[h]$. Furthermore, we can use the Lipschitz estimates from Corollary \ref{COR:LIP} and Proposition \ref{PROP:ADJ:LIP} to conclude that 
	\begin{align*}
		\kappa u_k -\phi_{u_k}\, \h(\varphi_{u_k}) \;\to\; \kappa \u -\phi_{\u}\, \h(\varphi_{\u}) \quad \text{in }L^2(L^2)\quad \text{as }k\to \infty.
	\end{align*}
	Now, since $(h_k)$ is uniformly bounded in $L^2(L^2)$, we obtain that
	\begin{align*}
		\big| J'(u_k)[h_k] - J'(\u)[h_k] \big| 
			\le \big\|	\kappa u_k -\phi_{u_k}\, \h(\varphi_{u_k}) - \kappa \u +\phi_{\u}\, \h(\varphi_{\u})\big\|_{2}
			\, \big\|h_k\big\|_{2} \to 0 \quad \text{as}\; k\to\infty.
	\end{align*}
	Consequently,
	\begin{align*}
		J'(u_k)[h_k] - J'(\u)[h] = J'(u_k)[h_k] - J'(\u)[h_k] + J'(\u)[h_k] - J'(\u)[h] \to 0 
		\quad \text{as}\; k\to\infty
	\end{align*}
	which proves (i).
	\paragraph{Proof of \textnormal{(ii)}:} The proof is very similar to the proof of (i). Let $(h_k)\subset L^2(L^2)$ with $h_k\wto h$ in $L^2(L^2)$ be any sequence. As $\phi'_\u[h] \,\h(\varphi_\u) + \phi_\u \,\h'(\varphi_\u) \varphi'_\u[h]$ lies in $L^2(L^2)$, we have $\B(h,h_k) \to \B(h,h)$. Moreover, due to Proposition \ref{LEM:FRE}, Proposition \ref{LEM:FREADJ} and the compact embeddings 
	\begin{align*}
		H^1(L^2)\cap L^\infty(H^2) \hookrightarrow C(\overline{\OmegaT}) \tand H^1((H^1)^*)\cap L^2(H^3) \hookrightarrow L^2(L^2),
	\end{align*}
	we obtain that
	\begin{align*}
		\big\|\varphi_\u'[h_k] - \varphi_\u'[h]\big\|_\infty \to 0 
		\tand 
		\big\|\phi_\u'[h_k] - \phi_\u'[h]\big\|_2 \to 0,
		\quad\text{as}\; k\to\infty,
	\end{align*}
	after extraction of a subsequence. Hence, we can conclude that
	\begin{align*}
		\phi'_\u[h_k] \,\h(\varphi_\u) + \phi_\u \,\h'(\varphi_\u) \varphi'_\u[h_k] 
			\;\to\; \phi'_\u[h] \,\h(\varphi_\u) + \phi_\u \,\h'(\varphi_\u) \varphi'_\u[h]
			\qquad \text{in $L^2(L^2)$ as $k\to \infty$}
	\end{align*}
	by means of Hölder's inequality. As $(h_k)$ is a bounded sequence in $L^2(L^2)$, it follows that 
	\begin{align*}
		&\big| \B(h_k,h_k) \to \B(h,h_k)  \big| \\
		&\quad \le \big\|	\phi'_\u[h_k] \,\h(\varphi_\u) + \phi_\u \,\h'(\varphi_\u) \varphi'_\u[h_k] 
			- 	\phi'_\u[h] \,\h(\varphi_\u) - \phi_\u \,\h'(\varphi_\u) \varphi'_\u[h] \big\|_{2} 
			\,\big\|h_k\big\|_{2} 
		\;\to\; 0 \quad \text{as}\; k\to\infty
	\end{align*}
	and thus 
	\begin{align*}
		\B(h_k,h_k) - \B(h,h) = \B(h_k,h_k) - \B(h,h_k) + \B(h,h_k) - \B(h,h) \to 0 \quad \text{as}\; k\to\infty
	\end{align*}
	which proves (ii).
\end{proof}

\subsection{A condition for global optimality of critical controls}
Even if a control $\u\in\U$ satisfies the sufficient optimality condition from Theorem \ref{THM:SUFF} it is absolutely not clear whether this control is globally optimal. However, in this section, we will establish a globality criterion for controls which satisfy the variational inequality or the projection formula, respectively. In the following, these controls will be referred to as \textbf{critical controls}. \\[1ex] 
The technique we are using was firstly introduced in \cite{AliDeckelnickHinze} for optimal control problems constrained by a general semilinear elliptic PDE of second order. Recently it has also been adapted for optimal control of the obstacle problem, see \cite{AliDeckelnickHinze2}. Our globality condition will be proved similarly and reads as follows: 

\begin{theorem}
	\label{THM:GLOB}
	Suppose that $\upalpha_1 > 0$ and $\kappa > 0$ and let $C_1$ and $L_1$ denote the constants from Proposition \ref{THM:EXS} and Corollary \ref{COR:LIP}. Moreover, we set
	\begin{align*}
		r := \underset{u\in\U}{\sup} \norm{\varphi_u}_\infty \le C_1.
	\end{align*}
	We assume that the control $\u\in\U$ satisfies the variational inequality \eqref{VIQ} (or the projection formula \eqref{PRF}, respectively) and that one of the following conditions holds:
	\begin{itemize}
		\item[\textnormal{(G1)}] It holds that
		\begin{align}
			\label{COND:G1}
			\frac \kappa 2 &\ge \Big[ \big\|(P\sigma_\u - A)(\phi_\u - q_\u) - \sigma_\u\rho_\u - \u\phi_\u \big\|_1
			\,\norm{\h''}_{L^\infty(\R)}\,L_1^2 \notag\\
			&\qquad\quad + \norm{\tau_\u}_1\, \norm{\psi'''}_{L^\infty([-r,r])} L_1^2 
			+ \norm{\phi_\u}_2 \,\norm{\h'}_{L^\infty(\R)}\,L_1 \Big].
		\end{align}
		\item[\textnormal{(G2)}] There exists a real number $\theta>0$ such that
		\begin{align}
			\label{COND:G2:1}
			2\kappa\theta \ge \norml{\infty}{\phi_\u}^2 \,\norm{\h''}_{L^\infty(\R)} 
		\end{align}
		and
		\begin{align}
		\label{COND:G2:2}
		\begin{aligned}
			\frac {\upalpha_1} 2 \;\ge\; &\big\|(P\sigma_\u - A)(\phi_\u - q_\u) - \sigma_\u\rho_\u - \u\phi_\u \big\|_\infty  \norm{\h''}_{L^\infty(\R)}   \\
			&+ \norm{\tau_\u}_\infty \norm{\psi'''}_{L^\infty([-r,r])} + \theta \norm{\phi_\u}_\infty^2 \norm{\h'}_{L^\infty(\R)}^2.
		\end{aligned}
		\end{align}
	\end{itemize}
	Then $\u$ is a globally optimal control of problem \eqref{OCP}. \\[1ex]
	In addition, the globally optimal control $\u$ is unique if one of the following conditions holds:
	\begin{itemize}
		\item[\textnormal{(U1)}] Condition \textnormal{(G1)} is satisfied and \eqref{COND:G1} holds with ``$\,>$" instead of ``$\,\ge$".
		\item[\textnormal{(U2)}] Condition \textnormal{(G2)} is satisfied and \eqref{COND:G2:1} holds with ``$\,>$" instead of ``$\,\ge$".
	\end{itemize}
\end{theorem}

\begin{remark}$\;$ \label{REM:UNI} 
	\begin{enumerate}
		\itema Of course, for the double-well potential $\psi$, we have $\psi'''(s)=6s$ and thus $\norm{\psi'''}_{L^\infty([-r,r])} = 6r$.
		\itemb The conditions (G1) and (G2) will be satisfied if the adjoint variables $\phi_\u$, $\tau_\u$, $\rho_\u$ and $q_\u$ are sufficiently small in the occurring norms. 
		\itemc Since the state and adjoint variables are sufficiently regular, the right-hand side of \eqref{COND:G1} is at least always finite. However, is seems very difficult to verify the condition (G1) by numerical methods as the Lipschitz constant $L_1$ which depends on the domain $\Omega$ has to be determined.
		\itemd Condition (G2) has the advantage that all occurring quantities except for $\norm{\psi'''}_{L^\infty([-r,r])}$ can be computed very easily.
		However, the constant $r$ can hardly be determined explicitly. \\[1ex]
		To overcome this disadvantage, one can use a modified version $\psi_\delta$ of the double-well potential such that $\psi_\delta\in C^\infty(\R)$ with $\psi = \psi_\delta$ on $[-\delta,\delta]$ 
		for some $\delta>1$ and $\psi_\delta'''$ bounded (cf. Example \ref{EX:POT}). It is not difficult to see that all other results in this paper remain true after this replacement (cf. Remark \ref{REM:PSI_H}(a)). \\[1ex]
		Of course, if $\delta> r$ the values of the state and costate variables will not change if $\psi$ is replaced by $\psi_\delta$. Various numerical results for the Cahn--Hilliard equation have shown that $1\le r \ll 2$ can be expected, i.e., $r$ is usually very close to one (see, e.g., \cite{HawkinsZeeKristofferOdenTinsley}). 
		\\[1ex]
		The strategy to verify (G2) is as follows: First, find $\theta$ such that \eqref{COND:G2:2} holds with equality because then $\theta$ is chosen as large as possible. Then check whether \eqref{COND:G2:1} is satisfied.
	\end{enumerate}
\end{remark}

In light of Remark \ref{REM:UNI}(d) we demonstrate a possible construction of such a potential $\psi_\delta$ in the following example.

\begin{example}\normalfont
\label{EX:POT}
It is well known that the function
\begin{align*}
	\zeta: \R\to\R,\quad s \mapsto 
	\begin{cases}
		0, &\text{if}\; s\le 0, \\
		\exp(-s^{-2}), &\text{if}\; s> 0
	\end{cases}
\end{align*}
is smooth. For any $\delta>1$, we define the function
\begin{align*}
	\xi_\delta: [0,\infty) \to\R,\quad s \mapsto 
	\begin{cases}
		1, &\text{if}\; s\le \delta, \\
		 \zeta(1+s-\delta) \zeta(1+\delta-s) /  \zeta(1)^2 , &\text{if}\; \delta < s < \delta + 1,  \\
		0 &\text{if}\; s \ge \delta + 1.
	\end{cases}
\end{align*}
It is straightforward to check that $\xi_\delta$ is smooth with $0 \le \xi_\delta \le 1$. Now, we construct the approximate potential $\psi_\delta$ by
\begin{align*}
	\psi_\delta: [0,\infty) \to\R,\quad s \mapsto \frac 1 4 - \frac 1 2 s^2 + 6 \int_0^s \int_0^x \int_0^y z\, \xi_\delta(|z|) \, \mathrm dz\mathrm dy\mathrm dx.	
\end{align*}
One can easily see that $\psi_\delta(s) \to\psi(s)$ for all $s\in\R$ as $\delta$ tends to infinity. Of course, $\psi_\delta$ is smooth with $\psi_\delta'''(s) = 6s\, \xi_\delta(s)$ and it holds that $\psi_\delta = \psi$ on $[-\delta,\delta]$. Thus, we obtain the bound
\begin{align*}
	\norm{\psi_\delta'''}_{L^{\infty}([-r,r])} \le \norm{\psi_\delta'''}_{L^{\infty}(\R)} \le 6(\delta+1).
\end{align*}
This means that the term $\norm{\psi_\delta'''}_{L^{\infty}([-r,r])}$ in condition (G2) can simply be replaced by $6(\delta+1)$ and thus, explicit knowledge of the constant $r$ is no longer required. The plots in \autoref{FIG:1} and \autoref{FIG:2} show that $\psi_\delta$ is a good approximation to $\psi$ even if $\delta>1$ is very close to one. 

\begin{figure}[h!]
\begin{minipage}{0.49\textwidth}
	\centering
	\includegraphics[scale=0.55]{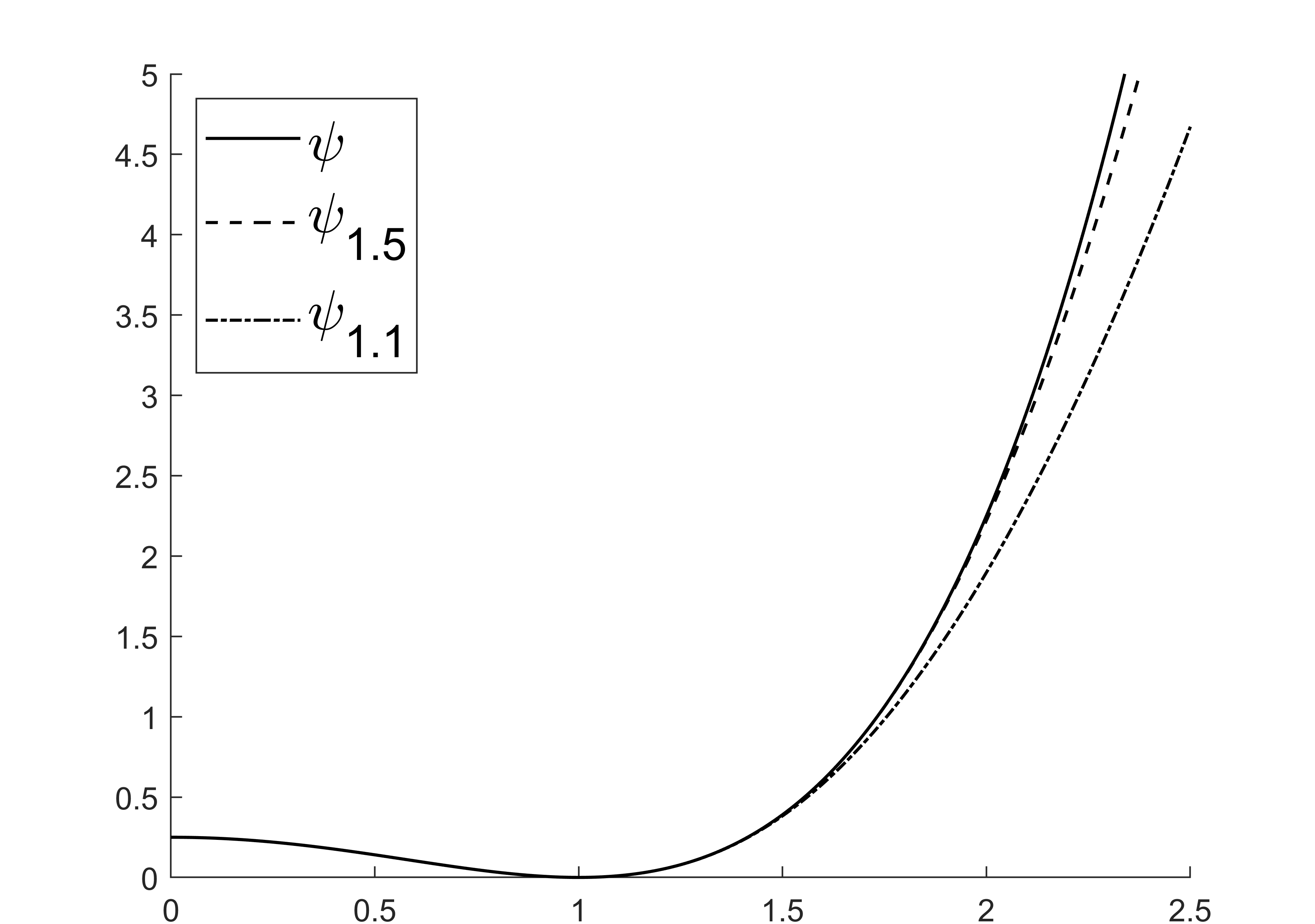}
	\caption{Plot of the double-well potential $\psi$ and \\ the approximate potentials $\psi_{1.5}$ and $\psi_{1.1}$.}
	\label{FIG:1}
\end{minipage}
\hfill
\begin{minipage}{0.49\textwidth}
	\centering
	\includegraphics[scale=0.55]{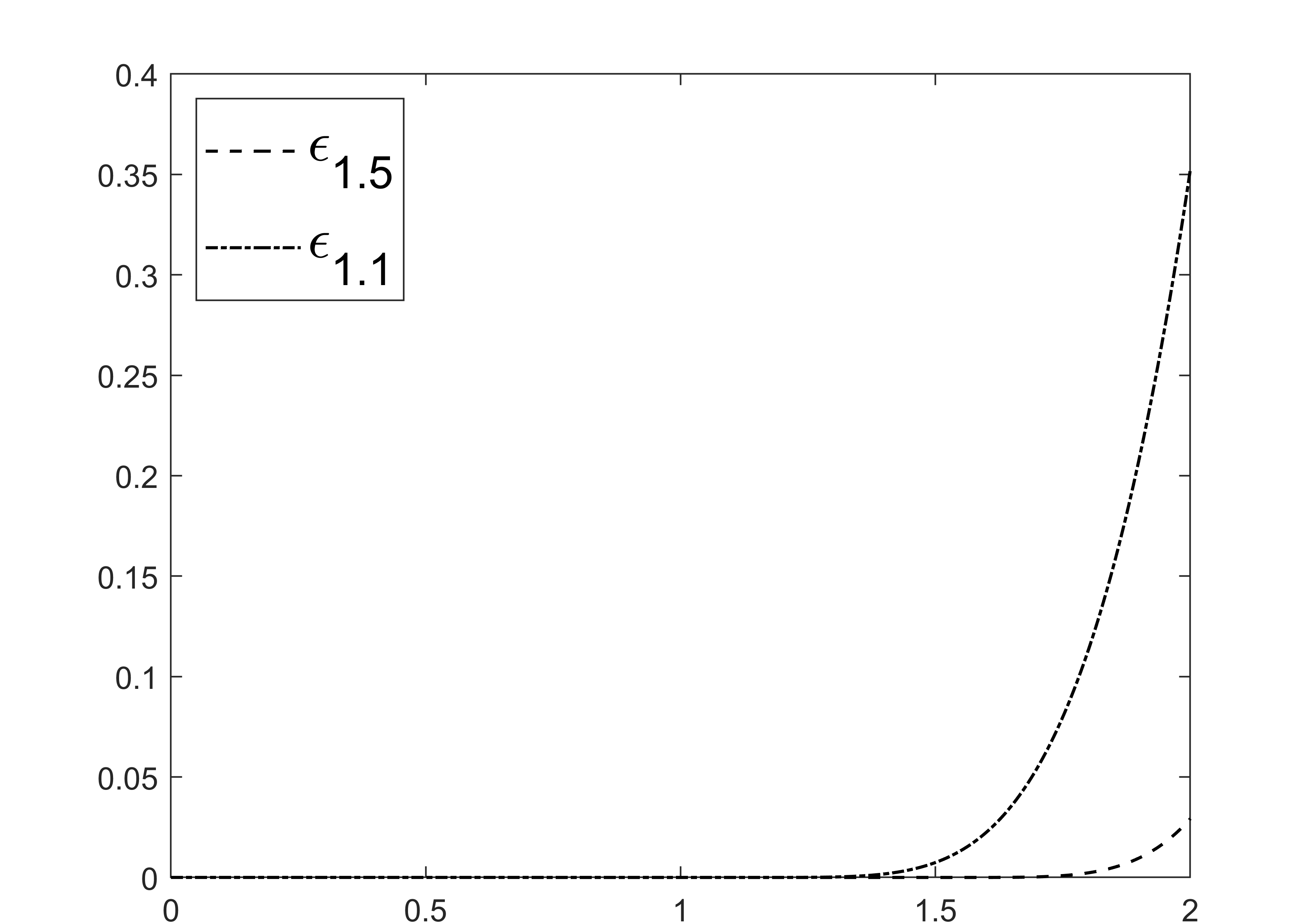}
	\caption{Plot of the error $\epsilon_{1.5} := \psi - \psi_{1.5} $ \\and $\epsilon_{1.1} := \psi - \psi_{1.1}$.}
	\label{FIG:2}
\end{minipage}
\end{figure}
\end{example}

\FloatBarrier

We will now present the proof of our theorem:

\begin{proof}[Proof of Theorem \ref{THM:GLOB}] To prove global optimality of the control $\u$, we intend to show that
	$J(u)-J(\u)\ge 0$ for all $u\in \U\setminus\{\u\}$. The proof is divided into three steps.
	\paragraph{Step 1:}  Let $u\in\U$ be arbitrary. Recall that 
	\begin{align*}
	\big( \kappa\u \,{,}\, u - \u \big)_{L^2(L^2)} \ge - \big( \phi_\u \,\h(\varphi_\u) \,{,}\, u - \u \big)_{L^2(L^2)}
	\end{align*}
	due to the variational inequality \eqref{VIQ}. Then, by a straightforward computation, we obtain that
	\begin{align}
	\label{EQ:GLOB0}
	J(u)-J(\u) 
	&\ge  \frac{\upalpha_0}{2} \norm{\varphi_u(T) - \varphi_\u(T)}_{L^2}^2
	+ \frac{\upalpha_1}{2} \norm{\varphi_u - \varphi_\u}_{L^2(L^2)}^2
	+ \frac{\kappa}{2} \norm{u - \u}_{L^2(L^2)}^2 + \CR,
	\end{align}
	where
	\begin{align*}
	\CR := \alpha_0 \big( \varphi_\u(T) - \varphi_f \,{,}\, \varphi_u(T) - \varphi_\u(T) \big)_{L^2} + \alpha_1 \big( \varphi_\u - \varphi_d \,{,}\, \varphi_u - \varphi_\u \big)_{L^2(L^2)} - \big( \phi_\u \,\h(\varphi_\u) \,{,}\, u - \u \big)_{L^2(L^2)}.
	\end{align*}
	Our aim is to show that
	\begin{align}
	\label{EQ:GLOB:R}
		|\CR|\le \frac{\upalpha_1}{2} \norm{\varphi_u - \varphi_\u}_{L^2(L^2)}^2 + \frac{\kappa}{2} \norm{u - \u}_{L^2(L^2)}^2
		\quad \text{for all}\; u\in\U\setminus\{\u\}
	\end{align}
	if condition (G1) or condition (G2) is fulfilled. Then \eqref{EQ:GLOB0} yields $J(u) \ge J(\u)$ and global optimality of $\u$ directly follows. 
	\paragraph{Step 2:} The idea is to express the remainder $\CR$ by the adjoint variables. For brevity, we write
	\begin{align*}
	(\tvarphi,\tmu,\tsigma,\tv,\tp) := (\varphi_u,\mu_u,\sigma_u,\v_u,p_u) - (\varphi_\u,\mu_\u,\sigma_\u,\v_\u,p_\u).
	\end{align*}
	This means that $(\tvarphi,\tmu,\tsigma,\tv,\tp)$ is a solution of \eqref{EQ:DIFF}. In the following, the strategy is to test the equations of the system \eqref{EQ:DIFF} with the adjoint variables. Testing \eqref{EQ:DIFF3} with $\phi_\u$ and integration by parts with respect to $t$ yields
	\begin{align*}
	0 &= \intOT \Big[\delt\tvarphi +\divergence(\varphi_u\tv) +\divergence(\tvarphi\v_\u) -m\laplace\tmu - P\tsigma\h(\varphi_u) - (P\sigma_\u-A)\big(\h(\varphi_u) - \h(\varphi_\u)\big)\\
	&\qquad\qquad + \big( u\h(\varphi_u) - \u\h(\varphi_\u) \big) \Big] \;\phi_\u \dxt \\
	&=  \intOT \divergence(\varphi_u\tv)\phi_\u +\divergence(\tvarphi\v_\u)\phi_\u -m\laplace\tmu\,\phi_\u - P\tsigma\h(\varphi_u)\phi_\u - (P\sigma_\u-A)\big(\h(\varphi_u) - \h(\varphi_\u)\big)\phi_\u \dxt \\
	&\quad + \intOT u\h(\varphi_u)\phi_\u - \u\h(\varphi_\u)\phi_\u \dxt   + \intO \tvarphi(T)\,\phi_\u(T) \dx - \int_0^T \langle \delt\phi_\u\,{,}\, \tvarphi \rangle_{H^1} \dt.
	\end{align*}
	Now, the term $\langle \delt\phi_\u\,{,}\, \tvarphi \rangle_{H^1}$ can be expressed by \eqref{WF:ADJ3} with test function $\tphi = \tvarphi$. We obtain that
	\begin{align}
	\label{EQ:GLOB1}
		0 &= \intOT \hspace{-8pt}\divergence(\varphi_u\tv)\phi_\u  
			- P\tsigma\h(\varphi_u)\phi_\u - (P\sigma_\u-A)\big(\h(\varphi_u) - \h(\varphi_\u)\big)\phi_\u  + u\h(\varphi_u)\phi_\u - \u\h(\varphi_\u)\phi_\u\dxt \notag\\
		&\quad + \intOT\hspace{-8pt} -m\laplace\tmu\,\phi_\u + (P\sigma_\u-A-\u)\h'(\varphi_\u)\phi_\u\tvarphi + \h'(\varphi_\u)\sigma_\u \rho_\u\tvarphi 
			+ \psi''(\varphi_\u)\tau_\u\tvarphi - (P\sigma_\u - A)\h'(\varphi_\u) q_\u\tvarphi \dxt \notag\\
		&\quad + \intOT \upalpha_1 (\varphi_\u - \varphi_d)\tvarphi + (\mu_\u + \chi\sigma_\u)\w_\u\cdot\grad\tvarphi +\grad\tau_\u\cdot\grad\tvarphi \dxt 
			+ \intO \tvarphi(T)\,\phi_\u(T) \dx.
	\end{align}
	Furthermore, testing \eqref{EQ:DIFF2} with $\w_\u$ yields
	\begin{align*}
		0 &= \intOT -\divergence\big( T(\tv,\tp) \big)\cdot \w_\u + \nu\tv\cdot\w_\u - (\tmu + \chi\tsigma) \grad\varphi_u \cdot \w_\u - (\mu_\u + \chi\sigma_\u) \grad\tvarphi \cdot \w_\u \dxt \\
		&= \intOT 2\eta \D\w_\u : \grad\tv + \nu\tv\cdot\w_\u - (\tmu + \chi\tsigma) \grad\varphi_u \cdot \w_\u - (\mu_\u + \chi\sigma_\u) \grad\tvarphi \cdot \w_\u \dxt	
	\end{align*}
	due to the definition of $T(\tv,\tp)$ and the fact that $\divergence(\w_\u)=0$. The term $2\eta \D\w_\u : \grad\tv$ can be expressed by choosing $\tw = \tv$ in \eqref{WF:ADJ2}. Thus,
	\begin{align}
	\label{EQ:GLOB2}
		0 &= \intOT \divergence(\tv)q_\u - \phi_\u \grad\varphi_u\cdot\tv - \phi_\u \varphi_u \divergence(\tv) - (\tmu + \chi\tsigma) \grad\varphi_u \cdot \w_\u 
			- (\mu_\u + \chi\sigma_\u) \grad\tvarphi \cdot \w_\u \dxt.
	\end{align}
	Proceeding similarly with the other subequations of \eqref{EQ:DIFF} gives
	\begin{align}
	\label{EQ:GLOB3}
		0&= \intOT \grad\varphi_\u\cdot\w_\u\,\tmu + m\laplace\tmu\,\phi_\u - \grad\tvarphi\cdot\grad\tau_\u - \big(\psi'(\varphi_u) - \psi'(\varphi_\u)\big)\tau_\u + \chi\tsigma\tau_\u \dxt, \\
	\label{EQ:GLOB4}
		0&= \intOT -\divergence(\tv)q_\u + P\tsigma\h(\varphi_u)q_\u + (P\sigma_u-A)\big( \h(\varphi_u) - \h(\varphi_\u) \big)q_\u \dxt, \\
	\label{EQ:GLOB5}
		0&= \intOT -\chi\tau_\u\tsigma + P\tsigma\h(\varphi_u)\phi_\u  + \chi\tsigma \grad\varphi_u\cdot\w_\u  - P\tsigma \h(\varphi_u)q_\u - \sigma_\u \rho_\u \big( \h(\varphi_u) - \h(\varphi_\u) \big) \dxt.
	\end{align}
	Adding up \eqref{EQ:GLOB1}-\eqref{EQ:GLOB5}, we ascertain that a large number of terms cancels out. We end up with
	\begin{align}
	\label{EQ:GLOB6}
		0 &= \intO \phi_\u(T) \tvarphi(T) \dx + \upalpha_1 \intOT (\varphi_\u - \varphi_d) \tvarphi \dxt + \intOT (u - \u)\h(\varphi_u)\phi_\u \dxt \notag \\
		&\quad + \intOT (P\sigma_\u - A) \big(\h(\varphi_u) - \h(\varphi_\u) - \h'(\varphi_u)\tvarphi \big)(q_\u - \phi_\u) \dxt \notag \\
		&\quad + \intOT \sigma_\u\rho_\u \big(\h(\varphi_u) - \h(\varphi_\u) - \h'(\varphi_u)\tvarphi \big)  \dxt \notag \\
		&\quad - \intOT \big( \psi'(\varphi_u) - \psi'(\varphi_\u) - \psi''(\varphi_\u)\tvarphi \big)\tau_\u \dxt \notag \\
		&\quad + \intOT \u\big( \h(\varphi_u) - \h(\varphi_\u) - \h'(\varphi_\u)\tvarphi \big)\phi_\u + (u-\u)\big( \h(\varphi_u) - \h(\varphi_\u) \big)\phi_\u \dxt.
	\end{align}
	Since $\phi_\u(T) = \upalpha(\varphi_\u(T) -\varphi_f)$, the first three terms on the right-hand side of \eqref{EQ:GLOB6} are equal to $\CR$. Moreover, using Taylor expansion, we can find $\zeta,\xi,\theta \in [0,1]$ such that 
	\begin{align*}
		\h(\varphi_u) - \h(\varphi_\u) &= \h'(\varphi_\zeta)\tvarphi^2 &&\twith \varphi_\zeta = \varphi_\u + \zeta(\varphi_u-\varphi_\u), \\
		\h(\varphi_u) - \h(\varphi_\u) - \h'(\varphi_u)\tvarphi &= \h''(\varphi_\xi)\tvarphi^2 &&\twith \varphi_\xi = \varphi_\u + \xi(\varphi_u-\varphi_\u), \\
		\psi'(\varphi_u) - \psi'(\varphi_\u) - \psi''(\varphi_u)\tvarphi &= \psi'''(\varphi_\theta)\tvarphi^2 &&\twith \varphi_\theta = \varphi_\u + \theta(\varphi_u-\varphi_\u).
	\end{align*}
	
	\pagebreak[2]
	Hence, it follows that 
	\begin{align}
	\label{EQ:GLOB7}
		\CR &= \intOT \Big[(P\sigma_\u - A)(\phi_\u - q_\u) - \sigma_\u\rho_\u - \u \phi_\u \Big] \h''(\varphi_\xi)\tvarphi^2  \dxt \notag \\
		&\quad + \intOT \tau_\u \psi'''(\varphi_\theta)\tvarphi^2 \dxt  + \intOT(u-\u) \phi_\u \h'(\varphi_\zeta)\tvarphi \dxt.
	\end{align}
	\paragraph{Step 3:} Now, we will use \eqref{EQ:GLOB7} to prove estimates in the fashion of \eqref{EQ:GLOB:R}. A simple computation gives
	\begin{align}
	\label{EQ:GLOB8}
		|\CR| &\le \Big[ \big\|(P\sigma_\u - A)(\phi_\u - q_\u) - \sigma_\u\rho_\u - \u\phi_\u \big\|_1
			\norm{\h''}_{L^\infty(\R)}\,L_1^2   \notag\\
		&\qquad + \norm{\tau_\u}_1\, \norm{\psi'''}_{L^\infty([-r,r])}\, L_1^2 
			+ \norm{\phi_\u}_2 \norm{\h'}_{L^\infty(\R)}\,L_1 \Big] \, \norm{u-\u}_2^2,
	\end{align}
	since $\norm{\varphi_\theta}_\infty \le r$. Furthermore, using Young's inequality with $\theta$, the remainder $\CR$ can also be bounded by
	\begin{align}
	\label{EQ:GLOB9}
		|\CR| &\le \Big[ \big\|(P\sigma_\u - A)(\phi_\u - q_\u) - \sigma_\u\rho_\u - \u\phi_\u \big\|_\infty  \norm{\h''}_{L^\infty(\R)}  \notag\\
		&\qquad\quad + \norm{\tau_\u}_\infty \norm{\psi'''}_{L^\infty([-r,r])} + \theta \norm{\phi_\u}_\infty^2 \norm{\h'}_{L^\infty(\R)}^2 \Big] \norm{\tvarphi}_{L^2(L^2)}^2 \notag\\
		&\qquad + \frac{1}{4\theta} \norm{\phi_\u}_\infty^2 \norm{\h'}_{L^\infty(\R)}^2 \norm{u-\u}_{L^2(L^2)}^2.
	\end{align}
	Hence, if condition (G1) or condition (G2) is satisfied, we can use \eqref{EQ:GLOB8} or \eqref{EQ:GLOB9} to conclude that
	\begin{align}
		|\CR| &\le \frac {\upalpha_1} 2 \norm{\tvarphi}_{L^2(L^2)}^2 + \frac \kappa 2 \norm{u-\u}_{L^2(L^2)}^2
	\end{align}
	and inequality \eqref{EQ:GLOB0} implies that $\u$ is a globally optimal control. If, in addition, either condition (U1) or (U2) is satisfied, it even holds that
	\begin{align}
		|\CR| &< \frac {\upalpha_1} 2 \norm{\tvarphi}_{L^2(L^2)}^2 + \frac \kappa 2 \norm{u-\u}_{L^2(L^2)}^2.
	\end{align}
	Then \eqref{EQ:GLOB0} implies that
	\begin{align*}
		J(u) > J(\u) \quad \text{for all}\; u\in\U\setminus\{\u\}
	\end{align*}
	and uniqueness of the globally optimal control $\u$ follows.
\end{proof}

\subsection{Uniqueness of the optimal control on small time intervals}

Finally, we present a condition on $T$ which ensures uniqueness of the optimal control. A similar result was established, e.g., in \cite{Knopf,KnopfWeber}. The idea behind the approach is as follows: If we choose the final time $T$ sufficiently small, the state equation will differ only slightly from its linearization. In the case $\kappa>0$, a linearized state equation would produce a strictly convex cost functional and the corresponding optimal control would be unique. If $T$ is small enough, we can expect that this property transfers to our problem. On the other hand, if the parameter $\kappa$ is large, the strictly convex part of the cost functional $J$ will be more dominant. Thus, it is not surprising that the size of the time interval on which the optimal control is unique will also depend on $\kappa$.\\[1ex]
In our theorem, we use the following notation: For any $p\in [1,6]$, let $c_\Omega(p) \ge 0$ denote a constant for which Sobolev's inequality
\begin{align*}
	\norm{v}_{L^p(\Omega)} \le c_\Omega(p)\, \norm{v}_{H^1(\Omega)}, \quad \text{for all}\; v\in H^1(\Omega),
\end{align*}
is satisfied.

\begin{theorem}
	\label{THM:UNILOC}
	Suppose that $\kappa>0$ and let $\u\in\U$ be a locally optimal control of problem \eqref{OCP}. Let $p,q \in [3,6]$ with $\frac 1 p + \frac 1 q= \frac 1 2$ be arbitrary.\\[1ex]
	Moreover, we assume that
	\begin{align}
	\label{ASS:UNIQ}
	T < \left(\frac{\sqrt{3}\,\kappa}
		{2 \big(L_3  + \sqrt{2} L_1  c_\Omega(p)\, c_\Omega(q) \|\phi_\u\|_{L^\infty(H^1)}\, \|\h'\|_{L^\infty(\R)}\big)}\right)^{4/3}.
	\end{align}
	Then, $\u$ is the unique locally optimal control.
\end{theorem}

\medskip

%

\begin{proof}
	Let us assume that there exists another locally optimal control $u$. Then, it holds that 
	\begin{align*}
		&\|\varphi_u(t) - \varphi_\u(t)\|_{L^2}^2 
		\le 2 \int_0^t \|\del_s\varphi_u(s) - \del_s\varphi_\u(s)\|_{L^2}\,
			\|\varphi_u(s) - \varphi_\u(s)\|_{L^2} \ds\\
		&\quad \le 2\sqrt{t}\,\|\varphi_u - \varphi_\u\|_{H^1(L^2)}\, 
			\|\varphi_u - \varphi_\u\|_{L^\infty(L^2)}
		\le 2 L_1^2 \sqrt{t}\, \|u-\u\|_{L^2(L^2)}^2.
	\end{align*}
	Using integration by parts, we also obtain the estimate
	\begin{align*}
		&\|\grad\varphi_u(t) - \grad\varphi_\u(t)\|_{L^2}^2 
		\le 2 \int_0^t \|\del_s\varphi_u(s) - \del_s\varphi_\u(s)\|_{L^2}\,
			\|\laplace \varphi_u(s) - \laplace\varphi_\u(s)\|_{L^2} \ds\\
		&\quad \le 2\sqrt{t}\,\|\varphi_u - \varphi_\u\|_{H^1(L^2)}\, 
			\|\varphi_u - 	\varphi_\u\|_{L^\infty(H^2)}
		\le 2 L_1^2 \sqrt{t}\, \|u-\u\|_{L^2(L^2)}^2.
	\end{align*}
	Consequently, we have
	\begin{align}
		\label{EQ:UNIQ1}
		\|\varphi_u - \varphi_\u\|_{L^2(H^1)} 
			\le \frac{2\sqrt{2} }{\sqrt{3}}\, L_1\, T^{3/4}\,\|u-\u\|_{L^2(L^2)}.
	\end{align}
	In the same fashion, we derive the estimate
	\begin{align*}
	&\|\phi_u(t) - \phi_\u(t)\|_{L^2}^2 
		= 2 \int_t^T \langle \del_s\phi_u(s) - \del_s\phi_\u(s) \,{,}\, \phi_u(s) - \phi_\u(s) \rangle_{H^1} \ds\\
	&\quad \le 2 \sqrt{T-t}\, \|\phi_u - \phi_\u\|_{H^1((H^1)^*)}\, \|\phi_u - \phi_\u\|_{L^\infty(H^1)}
	\le 2 L_3^2 \sqrt{T-t}\, \|u-\u\|_{L^2(L^2)}^2
	\end{align*}
	and thus,
	\begin{align}
	\label{EQ:UNIQ2}
		\|\phi_u - \phi_\u\|_{L^2(L^2)} 
			\le \frac{2}{\sqrt{3}}\, L_3\, T^{3/4}\,\|u-\u\|_{L^2(L^2)}.
	\end{align}
	Furthermore, we know from Corollary \ref{COR:PRF} that both $u$ and $\u$ satisfy the projection formula \eqref{PRF}. A straightforward computation yields
	\begin{align*}
	|u(x,t)-\u(x,t)| 
	\le \frac 1 \kappa \|\h\|_{L^\infty(\R)} |\phi_u(x,t)-\phi_\u(x,t)|
	+ |\phi_\u(x,t)|\, \|\h'\|_{L^\infty(\R)} |\varphi_u(x,t)-\varphi_\u(x,t)|
	\end{align*}
	for almost all $(x,t)\in\OmegaT$. Recall that $\|\h\|_{L^\infty(\R)}\le 1$. Hence, using \eqref{EQ:UNIQ1} and \eqref{EQ:UNIQ2}, we conclude that
	\begin{align*}
	\norm{u-\u}_{L^2(L^2)} 
	&\le \frac 1 \kappa \|\phi_u-\phi_\u\|_{L^2(L^2)} 
		+ \frac 1 \kappa  \|\phi_\u\|_{L^\infty(L^p)}\, \|\h'\|_{L^\infty(\R)} \, \|\varphi_u-\varphi_\u\|_{L^2(L^q)} \\
	&\le \frac 1 \kappa \|\phi_u-\phi_\u\|_{L^2(L^2)} 
		+ \frac {1} \kappa c_\Omega(p) \|\phi_\u\|_{L^\infty(H^1)}\, \|\h'\|_{L^\infty(\R)} \, c_\Omega(q)\|\varphi_u-\varphi_\u\|_{L^2(H^1)} \\
	&\le \frac{2 }{\sqrt{3}\kappa} \; T^{3/4}\; 
		\Big(  L_3  + \sqrt{2}  L_1  c_\Omega(p)\, c_\Omega(q) \|\phi_\u\|_{L^\infty(H^1)}\, \|\h'\|_{L^\infty(\R)} \Big)\, \|u-\u\|_{L^2(L^2)}.
	\end{align*}
	However, if \eqref{ASS:UNIQ} is satisfied, we have
	\begin{align*}
		\frac{2}{\sqrt{3}\kappa} \; T^{3/4}\; 
		\Big(  L_3  + \sqrt{2}  L_1  c_\Omega(p)\, c_\Omega(q) \|\phi_\u\|_{L^\infty(H^1)}\, \|\h'\|_{L^\infty(\R)} \Big) < 1.
	\end{align*}
	 Therefore, the above inequality can hold true only if $\|u-\u\|_{L^2(L^2)}=0$ which means uniqueness of the locally optimal control.
\end{proof}

\begin{corollary}
	Suppose that $T>0$ and $\kappa>0$ satisfy the assumption \eqref{ASS:UNIQ} of Theorem \ref{THM:UNILOC}. Then, there exists a unique globally optimal control $\u$ of problem \eqref{OCP}.\\[1ex]
	In this case, each of the equivalent necessary optimality conditions \eqref{VIQ} and \eqref{PRF} is a sufficient condition for global optimality.
\end{corollary}

\begin{proof}
	Theorem \ref{THM:EXC} ensures the existence of at least one globally optimal control $\u\in\U$. Of course, $\u$ is also locally optimal. Hence, since assumption \eqref{ASS:UNIQ} holds, Theorem \ref{THM:UNILOC} implies that $\u$ is the unique locally optimal control. It follows immediately that $\u$ is the unique globally optimal control.\\[1ex]
	Moreover, $\u$ satisfies the equivalent necessary optimality conditions \eqref{VIQ} and \eqref{PRF}. Because of Theorem \ref{THM:UNILOC} it is also the only control satisfying these conditions. Hence, \eqref{VIQ} (or \eqref{PRF}, respectively) is a sufficent condition for global optimality. 
\end{proof}

\section*{Acknowledgements}

Matthias Ebenbeck was supported by the RTG 2339 ``Interfaces, Complex Structures, and Singular Limits" of the
German Science Foundation (DFG). The support is gratefully acknowledged.


\bigskip


\footnotesize
\bibliographystyle{plain}
\bibliography{oc-chb-2}


\end{document}